\documentclass[EJP]{ejpecp} % replace ECP by EJP if needed.

\usepackage[T1]{fontenc}
\usepackage{lmodern}

\SHORTTITLE{CONVENIENT TAIL BOUNDS FOR SUMS OF RANDOM TENSORS}

\TITLE{CONVENIENT TAIL BOUNDS FOR SUMS OF RANDOM TENSORS} % \thanks is optional. Insert line breaks with \\

\AUTHORS{%
   Shih Yu Chang\footnote{San Jose State University, United States of America.
    \EMAIL{shihyu.chang@sjsu.edu}}}%AUTHORS

 % Separate items with ;

\AMSSUBJ{15A52;15A72;49J55;60H25} % Edit. Separate items with ;
\VOLUME{0}
\YEAR{2020}
\PAPERNUM{0}
\DOI{10.1214/YY-TN}

\ABSTRACT{This work prepares new probability bounds for sums of random, independent, Hermitian tensors. These probability bounds characterize large-deviation behavior of the extreme eigenvalue of the sums of random tensors. We extend Lapalace transform method and Lieb's concavity theorem from matrices to tensors, and apply these tools to generalize the classical bounds associated with the names Chernoff, Bennett, and Bernstein from the scalar to the tensor setting. Tail bounds for the norm of a sum of random rectangular tensors are also derived from corollaries of random Hermitian tensors cases. The proof mechanism can also be applied to tensor-valued martingales and tensor-based Azuma, Hoeffding and McDiarmid inequalities are established. }

\newcommand{\define}{\stackrel{\mbox{\tiny def}}{=}}

\begin{document}

\section{Introduction}\label{sec:Introduction}

\subsection{From Random Matrices to Radom Tensors}\label{sec:From Random Matrices to Radom Tensors}

% The problem,  

% random matrices -> random tensors (from my paper why tensor)
% Compare Bernstein Concentration Inequalities for Tensors
%via Einstein Products
%Ziyan Luo∗,
%Liqun Qi†,
%and Philippe L. Toint‡, Hermitian Case

A random matrix is a matrix-valued random variable—that is, a matrix in which some or all entries are random variables. Random matrices have played an important role in computational mathematics~\cite{MR41539}, physics~\cite{MR730191}, neuroscience~\cite{wainrib2013topological}, wireless communication~\cite{tulino2004random}, control~\cite{turnovsky1974stability}, etc. Many important properties of scientific and engineering systems can be modelled mathematically as matrix problems. In order to consider a high-dimensional system, it is often more convenient to consider tensors, or \emph{multidimensional data}, instead of matrices (two-dimensional data). 

Tensors have various applications in science and engineering~\cite{kolda2009tensor}. In numerical applications, tensors can be applied to solve multilinear system of equations~\cite{wang2019neural}, high-dimensional data fitting/regression~\cite{MR3395816}, tensor complementary problem~\cite{MR3947912}, tensor eigenvalue problem~\cite{MR3479021}, etc. In data processing applications, tensor theory applications include unsupervised separation of unknown mixtures of data signals~\cite{wu2010robust, mirsamadi2016generalized}, data filtering~\cite{muti2007survey}, MIMO (multi-input multi-output) code-division~\cite{de2008constrained, chen2011new}, radar, passive sensing, and communications~\cite{nion2010tensor, sidiropoulos2000parallel}. In other applications, tensors are also utilized to characterize data with \emph{coupling effects}, e.g., network signal processing~\cite{shen2020topology, shen2017tensor, fu2015joint} and image processing~\cite{ko2020fast, jiang2020framelet}. Tensor decomposition methods have been reported recently to establish the latent-variable models, such as topic models in~\cite{anandkumar2015tensor}, and the connections between the orthogonal tensor decomposition and the method of moments for undertaking the \emph{Latent Dirichlet Allocation} (LDA) in~\cite{sidiropoulos2017tensor}. However, all these applications assume that systems modelled by tensors are fixed and such assumption is not true and practical in solving tensors associated issues. In recent years, there are more works begin to target theory about random tensors, see~\cite{MR3616422},~\cite{MR3783911},~\cite{MR4166204},
~\cite{MR4163785},~\cite{MR4140540}, and references therein. In this work, we will focus on establishing a series of tail bounds for sums of random tensors. 

\subsection{Technical Results}\label{sec:Technical Results} 

Given a finite sequence of random Hermitian tensors $\{ \mathcal{X}_i \} \in \mathbb{C}^{I_1 \times \cdots \times I_M \times I_1 \times \cdots \times I_M}$, the main purpose of this work is to bound the following probability:
\begin{eqnarray}\label{eq:the problem of max eigen}
\mathbb{P} \left( \lambda_{\max}\left( \sum\limits_{i} \mathcal{X}_i \right) \geq \theta    \right),
\end{eqnarray}
where $\lambda_{\max}$ represents the largest eigenvalue of a Hermitian tensor obtained from \emph{eigenvalue decomposition}. The problem posted in Eq.~\eqref{eq:the problem of max eigen} are associated to the following problems: (1) the smallest and the largest singular value of a sum of random tensors with square or rectangular shapes; (2) extension random variable probability bounds, e.g., Chernoff and Bernstein bounds, to tensors settings; (3) tensor martingales and other adapted random sequences of tensors. 

There are two main technical tools required by this work to build those tensor probability bounds. The first is \emph{Laplace transform method}, which provides a systematic way to give tail bounds for the sum of scalar random variables. In~\cite{MR1966716}, the authors apply Laplace transform method to bound Eq.~\eqref{eq:the problem of max eigen} with the matrix setting, i.e., the tail probability for the maximum eigenvalue of the sum of Hermitian matrices is controlled by a matrix version of the moment-generating function. They proved following:
\begin{eqnarray}\label{eq:Ahlswede Winter bound for max eigen}
\mathbb{P} \left( \lambda_{\max}\left( \sum\limits_{i} \mathbf{X}_i \right) \geq \theta    \right) \leq \inf\limits_{t > 0} \Big\{ e^{-t \theta} \mathbb{E} \mathrm{Tr} \exp \left( \sum\limits_{i} t \mathbf{X}_i \right)\Big\}.
\end{eqnarray}
However, the bound to $\inf\limits_{t > 0} \Big\{ e^{-t \theta} \mathbb{E} \mathrm{Tr} \exp \left( \sum\limits_{i} t \mathbf{X}_i \right)\Big\}$ is far from optimal according to~\cite{MR2946459} (see Sections 3.7 and 4.8 in~\cite{MR2946459}). In this work, we extend the Laplace transform method to tensors. The other important technique utilized in this work is to extend Lieb's concavity theorem~\cite{MR332080} to tensors. Tensor Lieb’s theorem is introduced in Section~\ref{sec:Tensor Trace Concavity}, and we illustrate how to combine this result with the tensor Laplace transform method as our main technique to prove a series of random tensors bounds. We utilize this mechanism as our main approach to derive a large family of probability inequalities that are essentially tight in a wide variety of scenarios. Most random matrix inequalities studied in~\cite{MR2946459} have same origin as those random tensor inequalities discussed at this work, however, we enhance main tools, i.e., Laplace transform method and Lieb's concavity theorem, from matrices to tensors, and obtain new random tensor inequalities with tensor orders and dimensions as factors. 

This work will derive various inequalities based on Eq.~\eqref{eq:the problem of max eigen}. We will provide basic notations first before listing main inequalities investigated at this paper. The symbols $\lambda_{\min}$ and $\lambda_{\max}$ are used to represent the minimum and maximum eigenvalues of a Hermitian tensor. The notation $\succeq$ is used to indicate the semidefinite ordering of tensors. $\left\Vert \mathcal{A} \right\Vert$ is defined as the largest singular value of the tensor $\mathcal{A}$. 
The first category of concentration is associated with a sum of real numbers multiplied by independent random variables. If the independent random variables are drawn as normal (Gaussian) distribution or Rademacher distribution, the sum of real numbers multiplied by such random variables has sub-Gaussian tail behavior. In Section~\ref{sec:Tensor Gaussian Series}, we will provide the proof and its application about the following theorem. 

\begin{theorem}[Tensor Gaussian and Normal Series]\label{thm:TensorGaussianNormalSeries_intro}
Given a finite sequence $\mathcal{A}_i$ of fixed Hermitian tensors with dimensions as $\mathbb{C}^{I_1 \times \cdots \times I_M \times I_1 \times \cdots \times I_M}$, and let $\{ \alpha_i \}$ be a finite sequence of independent normal variables. We define 
\begin{eqnarray}\label{eq:4_5_intro}
\sigma^2 &\define& \left\Vert \sum\limits_{i=1}^{n} \mathcal{A}^2_i \right\Vert,
\end{eqnarray}
then, for all $\theta \geq 0$, we have 
\begin{eqnarray}\label{eq:4_3_intro}
\mathrm{Pr}\left( \lambda_{\max} \left( \sum\limits_{i=1}^{n}\alpha_i \mathcal{A}_i \right) \geq \theta \right) \leq \mathbb{I}_1^M e^{-\frac{\theta^2 }{2 \sigma^2}},
\end{eqnarray}
and
\begin{eqnarray}\label{eq:4_4_intro}
\mathrm{Pr}\left( \left\Vert \sum\limits_{i=1}^{n}\alpha_i \mathcal{A}_i \right\Vert \geq \theta \right) \leq 2 \mathbb{I}_1^M e^{-\frac{\theta^2 }{2 \sigma^2}}.
\end{eqnarray}
This theorem is also valid for a finite sequence of independent Rademacher random variables $\{ \alpha_i \}$.
\end{theorem}

% Chernoff 

Chernoff bound provides an estimate on the probability of the concentration results related to the number of successes in a sequence of independent random trials. In the tensor situation, the similar theorem concerns a sum of positive-semidefinite random
tensors subject to a uniform eigenvalue bound. The tensor Chernoff bound indicates that the largest (or smallest) eigenvalues of the tensor series have similar binomial random variable behavior. 

\begin{theorem}[Tensor Chernoff Bound]\label{thm:TensorChernoffBoundII_intro}
Consider a sequence $\{ \mathcal{X}_i  \in \mathbb{C}^{I_1 \times \cdots \times I_M  \times I_1 \times \cdots \times I_M } \}$ of independent, random, Hermitian tensors that satisfy
\begin{eqnarray}
\mathcal{X}_i \succeq \mathcal{O} \mbox{~~and~~} \lambda_{\max}(\mathcal{X}_i) \leq T
\mbox{~~ almost surely.}
\end{eqnarray}
Define following two quantities:
\begin{eqnarray}
\mu_{\max} \define \lambda_{\max}\left( \sum\limits_{i=1}^{n} \mathbb{E} \mathcal{X}_i \right) \mbox{~~and~~} 
\mu_{\min} \define \lambda_{\min}\left( \sum\limits_{i=1}^{n} \mathbb{E} \mathcal{X}_i \right),
\end{eqnarray}
then, let $\mathbb{I}_1^M = \prod\limits_{m=1}^M I_m$, we have following two inequalities:
\begin{eqnarray}\label{eq:Chernoff II Upper Bound_intro}
\mathrm{Pr} \left( \lambda_{\max}\left( \sum\limits_{i=1}^{n} \mathcal{X}_i \right) \geq (1+\theta) \mu_{\max} \right) \leq \mathbb{I}_1^M \left(\frac{e^{\theta}}{ (1 + \theta)^{1 + \theta}  }\right)^{\mu_{\max}/T} ,\mbox{~~ for $\theta \geq 0$;}
\end{eqnarray}
and
\begin{eqnarray}\label{eq:Chernoff II Lower Bound_intro}
\mathrm{Pr} \left( \lambda_{\min}\left( \sum\limits_{i=1}^{n} \mathcal{X}_i \right) \leq (1 - \theta) \mu_{\min} \right) \leq \mathbb{I}_1^M \left(\frac{e^{-\theta}}{ (1 - \theta)^{1 - \theta}  }\right)^{\mu_{\min}/T} ,\mbox{~~ for $\theta \in [0,1]$.}
\end{eqnarray}
\end{theorem}
In Section~\ref{sec:Tensor Chernoff Bounds}, we will prove Chernoff inequality and discuss its applications. 

% Bernstein

Bernstein inequality is another inequality to bound the sum of independent, bounded random tensors by restricting the range of the maximum eigenvalue of each random tensors. Bernstein inequality can provide tighter bound compared to Hoeffding inequality. 
In Section~\ref{sec:Tensor Bernstein Bounds}, following tensor Bernstein bound theorem and its applications are provided.
\begin{theorem}[Bounded $\lambda_{\max}$ Tensor Bernstein Bounds]\label{thm:Bounded Tensor Bernstein_intro}
Given a finite sequence of independent Hermitian tensors $\{ \mathcal{X}_i  \in \mathbb{C}^{I_1 \times \cdots \times I_M  \times I_1 \times \cdots \times I_M } \}$ that satisfy
\begin{eqnarray}\label{eq1:thm:Bounded Tensor Bernstein_intro}
\mathbb{E} \mathcal{X}_i = 0 \mbox{~~and~~} \lambda_{\max}(\mathcal{X}_i) \leq T 
\mbox{~~almost surely.} 
\end{eqnarray}
Define the total variance $\sigma^2$ as: $\sigma^2 \define \left\Vert \sum\limits_i^n \mathbb{E} \left( \mathcal{X}^2_i \right) \right\Vert$.
Then, we have following inequalities:
\begin{eqnarray}\label{eq2:thm:Bounded Tensor Bernstein_intro}
\mathrm{Pr} \left( \lambda_{\max}\left( \sum\limits_{i=1}^{n} \mathcal{X}_i \right)\geq \theta \right) \leq \mathbb{I}_1^M \exp \left( \frac{-\theta^2/2}{\sigma^2 + T\theta/3}\right);
\end{eqnarray}
and
\begin{eqnarray}\label{eq3:thm:Bounded Tensor Bernstein_intro}
\mathrm{Pr} \left( \lambda_{\max}\left( \sum\limits_{i=1}^{n} \mathcal{X}_i \right)\geq \theta \right) \leq \mathbb{I}_1^M \exp \left( \frac{-3 \theta^2}{ 8 \sigma^2}\right)~~\mbox{for $\theta \leq \sigma^2/T$};
\end{eqnarray}
and
\begin{eqnarray}\label{eq4:thm:Bounded Tensor Bernstein_intro}
\mathrm{Pr} \left( \lambda_{\max}\left( \sum\limits_{i=1}^{n} \mathcal{X}_i \right)\geq \theta \right) \leq \mathbb{I}_1^M \exp \left( \frac{-3 \theta}{ 8 T } \right)~~\mbox{for $\theta \geq \sigma^2/T$}.
\end{eqnarray}
\end{theorem}

We also applied techniques developed at this work to tensor martingales in Section~\ref{sec:Martingale Deviation Bounds}. In probability theory, the Azuma inequality gives a concentration result for the values of martingales that have bounded differences. Tensor Azuma inequality is given below.
%
% Tensor Azuma
\begin{theorem}[Tensor Azuma Inequality]\label{thm:Tensor Azuma_intro}
Given a finite adapted sequence of Hermitian tensors $\{ \mathcal{X}_i  \in \mathbb{C}^{I_1 \times \cdots \times I_M  \times I_1 \times \cdots \times I_M } \}$ and a fixed sequence of Hermitian tensors $\{ \mathcal{A}_i \}$ that satisfy
\begin{eqnarray}\label{eq1:thm:Tensor Azuma_intro}
\mathbb{E}_{i-1} \mathcal{X}_i = 0 \mbox{~~and~~} \mathcal{X}^2_i \preceq  \mathcal{A}^2_i \mbox{almost surely}, 
\end{eqnarray}
where $i = 1,2,3,\cdots$. 

Define the total varaince $\sigma^2$ as: $\sigma^2 \define \left\Vert \sum\limits_i^n \mathcal{A}_i^2 \right\Vert$.
Then, we have following inequalities:
\begin{eqnarray}\label{eq2:thm:Tensor Azuma_intro}
\mathrm{Pr} \left( \lambda_{\max}\left( \sum\limits_{i=1}^{n} \mathcal{X}_i \right)\geq \theta \right) \leq \mathbb{I}_1^M e^{-\frac{\theta^2}{8 \sigma^2}}.
\end{eqnarray}
\end{theorem}

In probability theory, Hoeffding's inequality builds an upper bound on the probability that the sum of bounded independent random variables drifts away from its expected value by more than a certain amount~\cite{MR144363}. In this work, we generalize this result to the tensor setting by considering random tensors that satisfy semidefinite upper bounds in Section~\ref{sec:Tensor Martingale Deviation Bounds}. In the tensor situation, the maximum eigenvalue for the sum of tensors also have subgaussian behavior.
%
% Tensor Hoeffding
%
\begin{theorem}[Tensor Hoeffding Inequality]\label{thm:Tensor Hoeffding}
Given a finite sequence of independent Hermitian tensors $\{ \mathcal{X}_i  \in \mathbb{C}^{I_1 \times \cdots \times I_M  \times I_1 \times \cdots \times I_M } \}$ and a fixed sequence of Hermitian tensors $\{ \mathcal{A}_i \}$ that satisfy
\begin{eqnarray}\label{eq1:thm:Tensor Hoeffding}
\mathbb{E}\mathcal{X}_i = 0 \mbox{~~and~~} \mathcal{X}^2_i \preceq  \mathcal{A}^2_i \mbox{almost surely}, 
\end{eqnarray}
where $i = 1,2,3,\cdots$. 

Define the total variance $\sigma^2$ as: $\sigma^2 \define \left\Vert \sum\limits_i^n \mathcal{A}_i^2 \right\Vert$.
Then, we have following inequalities:
\begin{eqnarray}\label{eq2:thm:Tensor Hoeffding}
\mathrm{Pr} \left( \lambda_{\max}\left( \sum\limits_{i=1}^{n} \mathcal{X}_i \right)\geq \theta \right) \leq \mathbb{I}_1^M e^{-\frac{\theta^2}{8 \sigma^2}}.
\end{eqnarray}
\end{theorem}

McDiarmid inequality is crucial in determining the stability of machine learning algorithms by applying a simple concept: small change in training set will make small change in hypothesis. In Section~\ref{sec:Tensor Martingale Deviation Bounds}, we extend McDiarmid inequality from the scalar-valued function to the tensor-valued function as shown below. 
\begin{theorem}[Tensor McDiarmid Inequality]\label{thm:Tensor McDiarmid_intro}
Given a set of $n$ independent random variables, i.e. $\{X_i: i = 1,2,\cdots n\}$, and let 
$F$ be a Hermitian tensor-valued function that maps these $n$ random variables to a Hermitian tensor of dimension within $\mathbb{C}^{I_1 \times \cdots \times I_M  \times I_1 \times \cdots \times I_M }$. Consider a sequence of Hermitian tensors $\{ \mathcal{A}_i \}$ that satisfy
\begin{eqnarray}\label{eq1:thm:Tensor McDiarmid_intro}
\left( F(x_1,\cdots,x_i,\cdots,x_n)  -  F(x_1,\cdots,x'_i,\cdots,x_n) \right)^2 \preceq \mathcal{A}^2_i,
\end{eqnarray}
where $x_i, x'_i \in X_i$ and $1 \leq i \leq n$. Define the total variance $\sigma^2$ as: $\sigma^2 \define \left\Vert \sum\limits_i^n \mathcal{A}_i^2 \right\Vert $.
Then, we have following inequality:
% x-z, y 
\begin{eqnarray}\label{eq2:thm:Tensor McDiarmid_intro}
\mathrm{Pr} \left( \lambda_{\max}\left( F(x_1,\cdots,x_n) - \mathbb{E}F(x_1,\cdots,x_n )\right)\geq \theta \right) \leq \mathbb{I}_1^M e^{-\frac{\theta^2}{8 \sigma^2}}.
\end{eqnarray}
\end{theorem}

% Main Results 1.1, 1.2 from User Friendly paper 

\subsection{Paper Organization}\label{sec:Paper Organization}

The remaining part of this paper is organized as follows. Basic probability theory required by our proofs and tensor notations are provided in Section~\ref{sec:Preliminaries of Probability and Tensors}. Main technical tools for tensor tail bounds are discussed in  Section~\ref{sec:Trace Concavity Method}. Section~\ref{sec:Tensor Gaussian Series} utilizes Gaussian and Rademacher series as case studies to explain tensor probability inequalities. Tensor Chernoff bound and its applications are discussed in Section~\ref{sec:Tensor Chernoff Bounds}. In Section~\ref{sec:Tensor Bernstein Bounds}, Tensor Bernstein bound and its applications are provided. Several martingale results based on random tensor are presented in Section~\ref{sec:Martingale Deviation Bounds}. Concluding remarks are given by Section~\ref{sec:Conclusion}.

\section{Preliminaries of Tensor and Probability}\label{sec:Preliminaries of Probability and Tensors}

In this section, we will provide a brief introduction of tensors and probability required for our future theory development. More detailed exposition about tensors can be found at~\cite{MR3660696,MR3791481}. In~\cite{MR559726}, it introduces those basic concepts about probability and moment-generating function. 

\subsection{Fundamental of Tensor}\label{sec:Fundamental of Tensors}

% 2.1-2.6
% no 2.2

\subsubsection{Tensor Notations}\label{sec:Tensor Notations}

Throughout this work, scalars are represented by lower-case letters (e.g., $d$, $e$, $f$, $\ldots$), vectors by boldfaced lower-case letters (e.g., $\mathbf{d}$, $\mathbf{e}$, $\mathbf{f}$, $\ldots$), matrices by boldfaced capitalized letters (e.g., $\mathbf{D}$, $\mathbf{E}$, $\mathbf{F}$, $\ldots$), and tensors by calligraphic letters (e.g., $\mathcal{D}$, $\mathcal{E}$, $\mathcal{F}$, $\ldots$), respectively. Tensors are multiarrays of values which are higher-dimensional generalizations from vectors and matrices. Given a positive integer $N$, let $[N] \define \{1, 2, \cdots ,N\}$. An \emph{order-$N$ tensor} (or \emph{$N$-th order tensor}) denoted by $\mathcal{X} \define (a_{i_1, i_2, \cdots, i_N})$, where $1 \leq i_j = 1, 2, \ldots, I_j$ for $j \in [N]$, is a multidimensional array containing $I_1 \times I_2 \times \cdots \times I_{N}$ entries. Let $\mathbb{C}^{I_1 \times \cdots \times I_N}$ and $\mathbb{R}^{I_1 \times \cdots \times I_N}$ be the sets of the order-$N$ $I_1 \times \cdots \times I_N$ tensors over the complex field $\mathbb{C}$ and the real field $\mathbb{R}$, respectively. For example, $\mathcal{X} \in \mathbb{C}^{I_1 \times \cdots \times I_N}$ is an order-$N$ multiarray, where the first, second, ..., and $N$-th dimensions have $I_1$, $I_2$, $\ldots$, and $I_N$ entries, respectively. Thus, each entry of $\mathcal{X}$ can be represented by $a_{i_1, \cdots, i_N}$. For example, when $N = 3$, $\mathcal{X} \in \mathbb{C}^{I_1 \times I_2 \times I_3}$ is a third-order tensor containing entries $a_{i_1, i_2, i_3}$'s.

Without loss of generality, one can partition the dimensions of a tensor into two groups, say $M$ and $N$ dimensions, separately. Thus, for two order-($M$+$N$) tensors: $\mathcal{X} \define (a_{i_1, \cdots, i_M, j_1, \cdots,j_N}) \in \mathbb{C}^{I_1 \times \cdots \times I_M\times
J_1 \times \cdots \times J_N}$ and $\mathcal{Y} \define (b_{i_1, \cdots, i_M, j_1, \cdots,j_N}) \in \mathbb{C}^{I_1 \times \cdots \times I_M\times
J_1 \times \cdots \times J_N}$, according to~\cite{MR3913666}, the \emph{tensor addition} $\mathcal{X} + \mathcal{Y}\in \mathbb{C}^{I_1 \times \cdots \times I_M\times
J_1 \times \cdots \times J_N}$ is given by 
\begin{eqnarray}\label{eq: tensor addition definition}
(\mathcal{X} + \mathcal{Y} )_{i_1, \cdots, i_M, j_1 \times \cdots \times j_N} &\define&
 a_{i_1, \cdots, i_M, j_1 \times \cdots \times j_N} \nonumber \\
& &+ b_{i_1, \cdots, i_M, j_1 \times \cdots \times j_N}. 
\end{eqnarray}
On the other hand, for tensors $\mathcal{X} \define (a_{i_1, \cdots, i_M, j_1, \cdots,j_N}) \in \mathbb{C}^{I_1 \times \cdots \times I_M\times
J_1 \times \cdots \times J_N}$ and $\mathcal{Y} \define (b_{j_1, \cdots, j_N, k_1, \cdots,k_L}) \in \mathbb{C}^{J_1 \times \cdots \times J_N\times K_1 \times \cdots \times K_L}$, according to~\cite{MR3913666}, the \emph{Einstein product} (or simply referred to as \emph{tensor product} in this work) $\mathcal{X} \star_{N} \mathcal{Y} \in  \mathbb{C}^{I_1 \times \cdots \times I_M\times
K_1 \times \cdots \times K_L}$ is given by 
\begin{eqnarray}\label{eq: Einstein product definition}
\lefteqn{(\mathcal{X} \star_{N} \mathcal{Y} )_{i_1, \cdots, i_M,k_1 \times \cdots \times k_L} \define} \nonumber \\ &&\sum\limits_{j_1, \cdots, j_N} a_{i_1, \cdots, i_M, j_1, \cdots,j_N}b_{j_1, \cdots, j_N, k_1, \cdots,k_L}. 
\end{eqnarray}
Note that we will often abbreviate a tensor product $\mathcal{X} \star_{N} \mathcal{Y}$ to ``$\mathcal{X} \hspace{0.05cm}\mathcal{Y}$'' for notational simplicity in the rest of the paper. 
This tensor product will be reduced to the standard matrix multiplication as $L$ $=$ $M$ $=$ $N$ $=$ $1$. Other simplified situations can also be extended as tensor–vector product ($M >1$, $N=1$, and $L=0$) and tensor–matrix product ($M>1$ and $N=L=1$). In analogy to matrix analysis, we define some basic tensors and elementary tensor operations as follows. 

\begin{definition}\label{def: zero tensor}
A tensor whose entries are all zero is called a \emph{zero tensor}, denoted by $\mathcal{O}$. 
\end{definition}

\begin{definition}\label{def: identity tensor}
An \emph{identity tensor} $\mathcal{I} \in  \mathbb{C}^{I_1 \times \cdots \times I_N\times
J_1 \times \cdots \times J_N}$ is defined by 
\begin{eqnarray}\label{eq: identity tensor definition}
(\mathcal{I})_{i_1 \times \cdots \times i_N\times
j_1 \times \cdots \times j_N} \define \prod_{k = 1}^{N} \delta_{i_k, j_k},
\end{eqnarray}
where $\delta_{i_k, j_k} \define 1$ if $i_k  = j_k$; otherwise $\delta_{i_k, j_k} \define 0$.
\end{definition}
In order to define \emph{Hermitian} tensor, the \emph{conjugate transpose operation} (or \emph{Hermitian adjoint}) of a tensor is specified as follows.  
\begin{definition}\label{def: tensor conjugate transpose}
Given a tensor $\mathcal{X} \define (a_{i_1, \cdots, i_M, j_1, \cdots,j_N}) \in \mathbb{C}^{I_1 \times \cdots \times I_M\times J_1 \times \cdots \times J_N}$, its conjugate transpose, denoted by
$\mathcal{X}^{H}$, is defined by
\begin{eqnarray}\label{eq:tensor conjugate transpose definition}
(\mathcal{X}^H)_{ j_1, \cdots,j_N,i_1, \cdots, i_M}  \define  
\overline{a_{i_1, \cdots, i_M,j_1, \cdots,j_N}},
\end{eqnarray}
where the overline notion indicates the complex conjugate of the complex number $a_{i_1, \cdots, i_M,j_1, \cdots,j_N}$. If a tensor $\mathcal{X}$ satisfies $ \mathcal{X}^H = \mathcal{X}$, then $\mathcal{X}$ is a \emph{Hermitian tensor}. 
\end{definition}
\begin{definition}\label{def: unitary tensor}
Given a tensor $\mathcal{U} \define (u_{i_1, \cdots, i_M, i_1, \cdots,i_M}) \in \mathbb{C}^{I_1 \times \cdots \times I_M\times I_1 \times \cdots \times I_M}$, if
\begin{eqnarray}\label{eq:unitary tensor definition}
\mathcal{U}^H \star_M \mathcal{U} = \mathcal{U} \star_M \mathcal{U}^H = \mathcal{I} \in \mathbb{C}^{I_1 \times \cdots \times I_M\times I_1 \times \cdots \times I_M},
\end{eqnarray}
then $\mathcal{U}$ is a \emph{unitary tensor}. 
\end{definition}
In this work, the symbol $\mathcal{U}$ is reserved for a unitary tensor. 

\begin{definition}\label{def: inverse of a tensor}
Given a \emph{square tensor} $\mathcal{X} \define (a_{i_1, \cdots, i_M, j_1, \cdots,j_M}) \in \mathbb{C}^{I_1 \times \cdots \times I_M\times I_1 \times \cdots \times I_M}$, if there exists $\mathcal{X} \in \mathbb{C}^{I_1 \times \cdots \times I_M\times I_1 \times \cdots \times I_M}$ such that 
\begin{eqnarray}\label{eq:tensor invertible definition}
\mathcal{X} \star_M \mathcal{X} = \mathcal{X} \star_M \mathcal{X} = \mathcal{I},
\end{eqnarray}
then $\mathcal{X}$ is the \emph{inverse} of $\mathcal{X}$. We usually write $\mathcal{X} \define \mathcal{X}^{-1}$ thereby. 
\end{definition}

We also list other crucial tensor operations here. The \emph{trace} of a square tensor is equivalent to the summation of all diagonal entries such that 
\begin{eqnarray}\label{eq: tensor trace def}
\mathrm{Tr}(\mathcal{X}) \define \sum\limits_{1 \leq i_j \leq I_j,\hspace{0.05cm}j \in [M]} \mathcal{X}_{i_1, \cdots, i_M,i_1, \cdots, i_M}.
\end{eqnarray}
The \emph{inner product} of two tensors $\mathcal{X}$, $\mathcal{Y} \in \mathbb{C}^{I_1 \times \cdots \times I_M\times J_1 \times \cdots \times J_N}$ is given by 
\begin{eqnarray}\label{eq: tensor inner product def}
\langle \mathcal{X}, \mathcal{Y} \rangle \define \mathrm{Tr}\left(\mathcal{X}^H \star_M \mathcal{Y}\right).
\end{eqnarray}
According to Eq.~\eqref{eq: tensor inner product def}, the \emph{Frobenius norm} of a tensor $\mathcal{X}$ is defined by 
\begin{eqnarray}\label{eq:Frobenius norm}
\left\Vert \mathcal{X} \right\Vert \define \sqrt{\langle \mathcal{X}, \mathcal{X} \rangle}.
\end{eqnarray}

We use $\lambda_{\min}$ and $\lambda_{\max}$ to represent the minimum and the maximum eigenvalues of a Hermitian tensor. The notation $\succeq$ is used to indicate the semidefinite ordering of tensors. If we have $\mathcal{X} \succeq \mathcal{Y}$, this means that the difference tensor $\mathcal{X} - \mathcal{Y}$ is a positive semidefinite tensor.

\subsubsection{Tensor Functions}\label{sec:Tensor Functions}

Given a function $g: \mathbb{R} \rightarrow \mathbb{R}$, the mapping result of a diagonal tensor by the function $g$ is to obtain another same size diagonal tensor with diagonal entry mapped by the function $g$. Then, the function $g$ can be extended to allow a Hermitian tensor $\mathcal{X} \in \mathbb{C}^{I_1 \times \cdots \times I_M\times I_1 \times \cdots \times I_M}$ as an input argument as
\begin{eqnarray}\label{eq:tensor func def}
g(\mathcal{X}) \define \mathcal{U} \star_M g(\Lambda) \star_M \mathcal{U}^H,~~\mbox{where $\mathcal{X} =  \mathcal{U}\star_M \Lambda \star_M \mathcal{U}^H $.}
\end{eqnarray}
The \emph{spectral mapping theorem} asserts that each eigenvalue of $g(\mathcal{X})$ is equal to $g(\lambda)$ for some eigenvalue $\lambda$ of $\mathcal{X}$. From the semidefinite ordering of tensors, we also have 
\begin{eqnarray}\label{eq:tensor psd ordering}
f(x) \geq g(x),~~\mbox{for $x \in [a, b]$}~~~ \Rightarrow ~~~ f(\mathcal{X})  \succeq 
 g(\mathcal{X}),~~\mbox{for eigenvalues of $\mathcal{X} \in [a, b]$;}
\end{eqnarray}
where $[a, b]$ is a real interval. 

\begin{definition}\label{def: tensor exponential}
Given a square tensor $\mathcal{X} \in \mathbb{C}^{I_1 \times \cdots \times I_M\times I_1 \times \cdots \times I_M}$, the \emph{tensor exponential} of the tensor $\mathcal{X}$ is defined as 
\begin{eqnarray}\label{eq: tensor exp def}
e^{\mathcal{X}} \define \sum\limits_{k=0}^{\infty} \frac{\mathcal{X}^{k}}{k !}, 
\end{eqnarray}
where $\mathcal{X}^0$ is defined as the identity tensor $\mathcal{I} \in \mathbb{C}^{I_1 \times \cdots \times I_M\times I_1 \times \cdots \times I_M}$ and \\
$\mathcal{X}^{k} = \underbrace{\mathcal{X} \star_M \mathcal{X} \star_M \dots \star_M\mathcal{X} }_{\mbox{$k$ terms of $\mathcal{X}$}} $.

Given a tensor $\mathcal{Y}$, the tensor $\mathcal{X}$ is said to be a \emph{tensor logarithm} of $\mathcal{Y}$ if $e^{\mathcal{X}}  = \mathcal{Y}$
\end{definition}

Several facts are about tensor exponential. First, the exponential of a Hermitian tensor is always positive-definite due to the spectral mapping theorem. Second, the trace exponential function, $\mathcal{X} \rightarrow \mathrm{Tr} \exp(\mathcal{X})$, is convex. Third, the trace exponential function follows monotone property with respect to semidefinite ordering as
\begin{eqnarray}\label{eq:trace exp monotone}
\mathcal{X} \succeq \mathcal{Y} ~~~ \Rightarrow ~~~ \mathrm{Tr} \exp(\mathcal{X}) \geq  \mathrm{Tr} \exp(\mathcal{Y}).
\end{eqnarray}
%
% proof see https://math.stackexchange.com/questions/2197309/convexity-of-matrix-trace-functions-a-mapsto-operatornametr-bfa
%
However, different from exponential rules for scalers, the tensor exponential does not convert sums into products. The Golden-Thompson inequality for tensors~\cite{2020arXiv201002152C} shows the following relationship
\begin{eqnarray}
\mathrm{Tr}e^{\mathcal{X} + \mathcal{Y}} \leq \mathrm{Tr} \left( e^{\mathcal{X} \star_M \mathcal{Y} }\right), 
\end{eqnarray}
where $\mathcal{X}, \mathcal{Y} \in \mathbb{C}^{I_1 \times \cdots \times I_M\times I_1 \times \cdots \times I_M}$ are Hermitian tensors.

For the tensor logarithm, we have the following monotone relation 
\begin{eqnarray}\label{eq:log monotone}
\mathcal{X} \succeq \mathcal{Y} ~~~ \Rightarrow ~~~ \log(\mathcal{X}) \succeq  \log(\mathcal{Y}).
\end{eqnarray}
Moreover, the tensor logarithm is also concave, i.e., we have
\begin{eqnarray}\label{eq:log concave}
t \log (\mathcal{X}_1) + (1-t) \log (\mathcal{X}_2)  \preceq \log( t \mathcal{X}_1 +
(1-t) \mathcal{X}_2 ),
\end{eqnarray}
where $\mathcal{X}_1, \mathcal{X}_2$ are positive-definite tensors and $t \in [0,1]$. The concavity of tensor logarithm can be derived from Hansen-Pedersen Characterizations, see~\cite{MR3426646}.

% paper at C:/Research/Random TEnsors/ A Survey on Operator Monotonicity, Operator Convexity, and Operator Means.pdf

If a given tensor $\mathcal{Y} \in \mathbb{C}^{I_1 \times \cdots \times I_M\times J_1 \times \cdots \times J_M}$ is not square, we can perform \emph{Hermitian dialation}, represented as $\mathbb{D}$, to the tensor $\mathcal{Y}$ as 
\begin{eqnarray}\label{eq:Hermitian dialation def}
\mathbb{D}(\mathcal{Y}) \define  \begin{bmatrix}
\mathcal{O} & \mathcal{Y} \\
\mathcal{Y}^H & \mathcal{O}  \\
\end{bmatrix}
\end{eqnarray}
where $\mathbb{D}(\mathcal{Y}) \in  \mathbb{C}^{(I_1 + J_1) \times \cdots \times (I_M + J_M) \times (I_1 + J_1) \times \cdots \times (I_M + J_M)}$ is a Hermitian tensor. Since we have 
\begin{eqnarray}
\mathbb{D}(\mathcal{Y})^2 =  \begin{bmatrix}
\mathcal{Y}\mathcal{Y}^H & \mathcal{O}   \\
\mathcal{O} & \mathcal{Y}^H \mathcal{Y}  \\
\end{bmatrix},
\end{eqnarray}
then, we have following spectral norm relation 
\begin{eqnarray}\label{eq:spectral identity}
\lambda_{\max}( \mathbb{D}(\mathcal{Y})  ) = \left\Vert \mathbb{D}(\mathcal{Y})\right\Vert = \left\Vert \mathcal{Y} \right\Vert,
\end{eqnarray}
where $\left \Vert \cdot \right \Vert$ is the spectral norm for a tensor and this will return the maximum singular value of its argument tensor. Hermitian dilation operation enables us to extend results from Hermitian tensors to other non-square tensors. 

\subsection{Tensor Moments and Cumulants}\label{sec:Tensor Moments and Cumulants}

Since the expectation of a random tensor can be treated as convex combination, expectation will preserve the semidefinite order as
\begin{eqnarray}
\mathcal{X} \succ \mathcal{Y} \mbox{~~almost surely} ~~~ \Rightarrow ~~~
\mathbb{E}(\mathcal{X}) \succ \mathbb{E}(\mathcal{Y}).
\end{eqnarray}
From operator Jensen's inequality~\cite{MR649196}, we also have 
\begin{eqnarray}
\mathbb{E}(\mathcal{X}^2) \succeq \left(\mathbb{E}(\mathcal{X})\right)^2.
\end{eqnarray}

% moment-generating function $\mathbb{M}_{\mathcal{X}}(t)$
% cumulant-generating function $\mathbb{K}_{\mathcal{X}}(t)$
% \theta to t 

Suppose a random Hermitian tensor $\mathcal{X}$ having tensor moments of all orders, i.e., $\mathbb{E}(\mathcal{X}^n)$ existing for all $n$, we can define the tensor moment-generating function, denoted as $\mathbb{M}_{\mathcal{X}}(t)$, and the tensor cumulant-generating function, denoted as $\mathbb{K}_{\mathcal{X}}(t)$, for the tensor $\mathcal{X}$ as 
\begin{eqnarray}\label{eq:def mgf and cgf}
\mathbb{M}_{\mathcal{X}}(t) \define \mathbb{E} e^{t \mathcal{X}}, \mbox{~~and~~~}
\mathbb{K}_{\mathcal{X}}(t) \define \log \mathbb{E} e^{t \mathcal{X}},
\end{eqnarray}
where $t \in \mathbb{R}$. Both the tensor moment-generating function and the tensor cumulant-generating function can be expressed as power series expansions:
\begin{eqnarray}\label{eq:mgf and cgf expans}
\mathbb{M}_{\mathcal{X}}(t) = \mathcal{I} + \sum\limits_{n=1}^{\infty} \frac{t^n}{n !}\mathbb{E}(\mathcal{X}^n), \mbox{~~and~~~}
\mathbb{K}_{\mathcal{X}}(t) = \sum\limits_{n=1}^{\infty} \frac{t^n}{n !} \psi_n,
\end{eqnarray}
where $\psi_n$ is called \emph{tensor cumulant}. The tensor cumulant $\psi_n$ can be expressed as a polynomial in terms of tensor moments up to the order $n$, for example, the first cumulant is the mean and the second cumulant is the variance: 
\begin{eqnarray}\label{eq:cumulant mean and var}
\psi_1 = \mathbb{E}(\mathcal{X}), \mbox{~~and~~~}
\psi_2 = \mathbb{E}(\mathcal{X}^2) - (\mathbb{E}(\mathcal{X}))^2.
\end{eqnarray}

Finally, in this work, we also assume that all random variables are sufficiently regular for us to compute their expectations, interchange limits, etc.

% 2.2
% 2.7
% 3.1
\section{Trace Concavity Method}\label{sec:Trace Concavity Method}

The main purpose of this section is to develop two important tools which will be applied 
intensively in the proof of probability inequalities for the maximum eigenvalue of a sum
of independent random tensors. The first one is the Laplace transform method for tensors discussed in Section~\ref{sec:Laplace Transform Method for Tensors}, and the second one is the tensor trace concavity which will be presented in Section~\ref{sec:Tensor Trace Concavity}. 

\subsection{Laplace Transform Method for Tensors}\label{sec:Laplace Transform Method for Tensors}

% 3.2
% \theta <-> t
We extend the Laplace transform bound from matrices to tensors based on~\cite{MR1966716}. Following lemma is given to establish the Laplace transform bound for tensors.
\begin{lemma}[Laplace Transform Method for Tensors]\label{lma: Laplace Transform Method}
Let $\mathcal{X}$ be a random Hermitian tensor. For $\theta \in \mathbb{R}$, we have
\begin{eqnarray}
\mathbb{P} (\lambda_{\max}(\mathcal{X}) \geq \theta) \leq \inf_{t > 0} \Big\{ e^{-\theta t} \mathbb{E}\mathrm{Tr} e^{t \mathcal{X}} \Big\}
\end{eqnarray}
\end{lemma}
\begin{proof}
Given a fix value $t$, we have 
\begin{eqnarray}\label{eq1: lma: Laplace Transform Method}
\mathbb{P} (\lambda_{\max}(\mathcal{X}) \geq \theta) = \mathbb{P} (\lambda_{\max}(t\mathcal{X}) \geq t \theta ) = \mathbb{P} (e^{ \lambda_{\max}(t\mathcal{X})  } \geq e^{ t \theta}  ) \leq e^{-t \theta} \mathbb{E} e^{ \lambda_{\max}(t\mathcal{X}) }. 
\end{eqnarray}
The first equality uses the homogeneity of the maximum eigenvalue map, the second equality comes from the monotonicity of the scalar exponential function, and the last relation is Markov's inequality. Because we have
\begin{eqnarray}\label{eq2: lma: Laplace Transform Method}
e^{\lambda_{\max} (t \mathcal{Y})} = \lambda_{\max}(e^{t \mathcal{Y}}) \leq  \mathrm{Tr} e^{t \mathcal{Y}},
\end{eqnarray}
where the first equality used the spectral mapping theorem, and the inequality holds because the exponential of an Hermitian tensor is positive definite and the maximum eigenvalue of a positive definite tensor is dominated by the trace~\cite{MR3913666}. From Eqs~\eqref{eq1: lma: Laplace Transform Method} and~\eqref{eq2: lma: Laplace Transform Method}, this lemma is established.
\end{proof}

The Lemma~\ref{lma: Laplace Transform Method} helps us to control the tail probabilities for the maximum eigenvalue of a random Hermitian tensor by utilizing a bound for the traces of the tensor moment-generating function introduced in Section~\ref{sec:Tensor Moments and Cumulants}.

\subsection{Tensor Trace Concavity}\label{sec:Tensor Trace Concavity}

% 3.3 Motivation
% 3.4 (Lieb concavity for tensors) 
% Using variational formula for the relative entropy (proved at my Multivariate Tensor
% Inequality paper) and Thm 2.13 from (paper Approximate quantum by David)
% For the join convexity of the relative entropy, 
% (1) using paper (A MATRIX CONVEXITY APPROACHTO SOME CELEBRATED QUANTUM INEQUALITIES) Thm 2.1, 2.2 and Cor 2.3
% (2) But, for Them 2.1 (Using the proof  sec. 4 in paper JENSEN’S OPERATOR INEQUALITY) 
% 3.5 Subadditivity of the tensor CGF
% 3.6 Tail bounds for ind. sums 

In this section, we will extend Lieb’s concavity theorem to tensors and we begin with the definition about the relative entropy between two tensors. 

\begin{definition}\label{def: relative entropy for tensors}
% , where the tensor $\mathcal{X}$ has the trace equal to one.
Given two positive definite tensors $\mathcal{A} \in \mathbb{C}^{I_1 \times \cdots \times I_M \times I_1 \times \cdots \times I_M}$ and tensor $\mathcal{B} \in \mathbb{C}^{I_1 \times \cdots \times I_M \times I_1 \times \cdots \times I_M}$. The relative entropy between tensors $\mathcal{A}$ and  $\mathcal{B}$ is defined as 
\begin{eqnarray}
D( \mathcal{A} \parallel \mathcal{B}) \define \mathrm{Tr} \mathcal{A} \star_M (\log \mathcal{A} - \log \mathcal{B}).
\end{eqnarray} 
\end{definition}

Given a continuous function defined over a real interval as $f: [a, b] \rightarrow \mathbb{R}$, and $\mathcal{T} \in \mathbb{C}^{I_1 \times \cdots \times I_M \times I_1 \times \cdots \times I_M}$ as a Hermitian tensor with spectrum in $[a, b]$ and spectrum decomposition as $\mathcal{T} = \sum\limits_{\lambda_n} \lambda_n \mathcal{U}_{\lambda_n}$, where $\mathcal{U}_{\lambda_n} \in \mathbb{C}^{I_1 \times \cdots \times I_M \times I_1 \times \cdots \times I_M}$ are mutually orthogonal tensors, then the mapping for the tensor $\mathcal{T}$ by $f$ can be defined as $f(\mathcal{T}) = \sum\limits_{\lambda_n} f(\lambda_n) \mathcal{U}_{\lambda_n}$. The function $f$ is called as a \emph{tensor convex} function if $f(\lambda_n)$ is convex on the Hermitian tenor with spectrum in $[a,b]$. We apply \emph{perspective function}~\cite{MR2475796} notion for tensor convex and introduce the following lemma about the convexity of a tensor convex function. 
\begin{lemma}\label{lma:perspective func}
Given $f$ as a tensor convex function, two commuting tensors $\mathcal{X}, \mathcal{Y} \in$ \\ $\mathbb{C}^{I_1 \times \cdots \times I_M \times I_1 \times \cdots \times I_M}$, i.e., $\mathcal{X}\star_M \mathcal{Y} = \mathcal{Y}\star_M \mathcal{X}$, and the existence of the $\mathcal{Y}^{-1}$, then the following map $h$ 
\begin{eqnarray}\label{eq:lma:perspective func}
h (\mathcal{X}, \mathcal{Y}) = f(\mathcal{X} \star_M \mathcal{Y}^{-1}) \star_M \mathcal{Y}
\end{eqnarray}
is jointly convex in the sense that, given $t \in [0, 1]$, if $\mathcal{X} = t \mathcal{X}_1 + (1-t) \mathcal{X}_2$ and $\mathcal{Y} = t \mathcal{Y}_1 + (1-t) \mathcal{Y}_2$ with $\mathcal{X}_1 \star_M \mathcal{Y}_1 = \mathcal{Y}_1 \star_M \mathcal{X}_1$ and $\mathcal{X}_2 \star_M \mathcal{Y}_2 = \mathcal{Y}_2 \star_M \mathcal{X}_2$, we should have 
\begin{eqnarray}
h (\mathcal{X}, \mathcal{Y})  \leq t h (\mathcal{X}_1, \mathcal{Y}_1) + (1-t) h (\mathcal{X}_2, \mathcal{Y}_2). 
\end{eqnarray}
\end{lemma}
\begin{proof}
Constructing tensors $\mathcal{A} = (t \mathcal{Y}_1)^{1/2} \star_M \mathcal{Y}^{-1/2}$
and $\mathcal{B} = ( (1-t)\mathcal{Y}_2)^{1/2} \star_M \mathcal{Y}^{-1/2}$, then we have 
\begin{eqnarray}\label{eq:lma:perspective func 1}
\mathcal{A}^H \star_M \mathcal{A} + \mathcal{B}^H \star_M \mathcal{B} = \mathcal{I}
\end{eqnarray}

Since we have 
\begin{eqnarray}
h(\mathcal{X}, \mathcal{Y}) &=& f(\mathcal{X} \star_M \mathcal{Y}^{-1} ) \star_M \mathcal{Y} \nonumber \\
&=&   \mathcal{Y}^{1/2} \star_M f(  \mathcal{Y}^{-1/2} \star_M \mathcal{X} \star_M \mathcal{Y}^{-1/2} ) \star_M \mathcal{Y}^{1/2} \nonumber \\
&=&  \mathcal{Y}^{1/2} \star_M f( \mathcal{A}^H \star_M \mathcal{X}_1 \star_M \mathcal{Y}^{-1}_1 \star_M \mathcal{A} +  \mathcal{B}^H \star_M  \mathcal{X}_2 \star_M \mathcal{Y}^{-1}_2  \star_M \mathcal{B} ) \star_M \mathcal{Y}^{1/2} \nonumber \\
&\leq_1& \mathcal{Y}^{1/2} \star_M \left( \mathcal{A}^H \star_M f(\mathcal{X}_1 \star_M \mathcal{Y}^{-1}_1 ) \star_M \mathcal{A} \right. \nonumber \\
&  & \left. +  \mathcal{B}^H \star_M f(\mathcal{X}_2 \star_M \mathcal{Y}^{-1}_2 )  \star_M \mathcal{B} \right) \star_M \mathcal{Y}^{1/2} \nonumber \\
&=&  (t \mathcal{Y}_1)^{1/2}  f(\mathcal{X}_1 \star_M \mathcal{Y}^{-1}_1 )  (t \mathcal{Y}_1)^{1/2} + ((1-t) \mathcal{Y}_2)^{1/2}  f(\mathcal{X}_2 \star_M \mathcal{Y}^{-1}_2 )  ((1-t) \mathcal{Y}_2)^{1/2}   \nonumber \\
&=& t h(\mathcal{X}_1, \mathcal{Y}_1)  + (1 - t) h(\mathcal{X}_2, \mathcal{Y}_2) 
\end{eqnarray}
where $\leq_1$ is based on Hansen-Pedersen-Jensen inequality and the condition provided by Eq.~\eqref{eq:lma:perspective func 1}, see Theorem 2.1 in~\cite{MR2475796}.    
\end{proof}

Following lemma is given to establish the joint convexity property of relative entropy for tensors.
\begin{lemma}[Joint Convexity of Relative Entropy for Tensors]\label{lma:joint conv of rela entropy}
The relative entropy function of two positive-definite tensors  is a jointly convex function. That is 
\begin{eqnarray}\label{eq:lma:joint conv of rela entropy}
\mathbb{D}(t \mathcal{A}_1 + (1 - t) \mathcal{A}_2 \parallel t \mathcal{B}_1 + (1 - t) \mathcal{B}_2) \leq t \mathbb{D}( \mathcal{A}_1 \parallel \mathcal{B}_1 ) +  (1- t) \mathbb{D} (\mathcal{A}_2 \parallel \mathcal{B}_2),
\end{eqnarray} 
where $ t \in [0, 1]$ and all the following four tensors $\mathcal{A}_1$, $\mathcal{B}_1$, $\mathcal{A}_2$ and $\mathcal{B}_2$, are positive definite. 
\end{lemma}
\begin{proof}
From the definition~\ref{def: relative entropy for tensors}, we wish to show the joint convexity of the function $\mathbb{D} (\mathcal{A} \parallel \mathcal{B})$ with respect to the tensors $\mathcal{A}, \mathcal{B} \in \mathbb{C}^{I_1 \times \cdots \times I_M \times I_1 \times \cdots \times I_M}$. Let us define tensor operators $\mathcal{F}(\mathcal{X}) \define \mathcal{A} \star_M \mathcal{X}$ and $\mathcal{G}(\mathcal{X}) \define \mathcal{X} \star_M \mathcal{B}$  for the variable tensor $\mathcal{X}  \in \mathbb{C}^{I_1 \times \cdots \times I_M \times I_1 \times \cdots \times I_M}$. Then, we have $\mathcal{F}(\mathcal{X})$ and $\mathcal{G}(\mathcal{X})$ commuting on the inner product operation $\langle \mathcal{F}(\mathcal{X}), \mathcal{G}(\mathcal{X}) \rangle$ defined by Eq.~\eqref{eq: tensor inner product def}, i.e., $\mathrm{Tr}(\mathcal{F}^H(\mathcal{X})\star_M \mathcal{G}(\mathcal{X}) ) =\mathrm{Tr}(\mathcal{G}^H(\mathcal{X}) \star_M \mathcal{F}(\mathcal{X}) )$. Since the function $f(x) = x \log x$ is tensor convex, we apply Lemma~\ref{lma:perspective func} to operators $\mathcal{F}(~~), \mathcal{G}(~~)$ and the function $h$ definition provided by Eq.~\eqref{eq:lma:perspective func} to obtain the following relation: 
\begin{eqnarray}
\langle \mathcal{I}, h (\mathcal{F}(\mathcal{I}), \mathcal{G}(\mathcal{I}) \rangle &=& 
\langle \mathcal{I},~~\mathcal{G}(\mathcal{I}) \star_M ( \mathcal{F}(\mathcal{I}) \star_M \mathcal{G}^{-1}(\mathcal{I} )) \log ( \mathcal{F}(\mathcal{I} ) \star_M \mathcal{G}^{-1}(\mathcal{I})  )   \rangle \nonumber \\
&=& \langle \mathcal{I}, \mathcal{F}(\mathcal{I}) (\log \mathcal{F}(\mathcal{I}) - \log \mathcal{G}(\mathcal{I}) )   \rangle \nonumber \\
&=& \mathrm{Tr}(\mathcal{A} \log \mathcal{A} - \mathcal{A} \log \mathcal{B} )= \mathbb{D}(\mathcal{A} \parallel \mathcal{B}),
\end{eqnarray}
is jointly convex with respect to tensors $\mathcal{A}$ and $\mathcal{B}$. 
\end{proof}

\begin{theorem}[Lieb's concavity theorem for tensors]\label{thm:Lieb concavity thm}
Let $\mathcal{H}$ be a Hermitian tensor. Following map 
\begin{eqnarray}\label{eq:thm:Lieb concavity thm}
\mathcal{A} \rightarrow \mathrm{Tr}e^{\mathcal{H}+ \log \mathcal{A}}
\end{eqnarray}
is concave on the positive-definite cone.
\end{theorem}
\begin{proof}

From Klein's inequality for the map $t \rightarrow t \log t$ (which is strictly concave for $t > 0$) and Hermitian tensors $\mathcal{X}, \mathcal{Y}, $we have 
\begin{eqnarray}
\mathrm{Tr} \mathcal{Y} \geq \mathrm{Tr} \mathcal{X} - \mathrm{Tr} \mathcal{X} \log \mathcal{X} + \mathrm{Tr} \mathcal{X} \log \mathcal{Y}.
\end{eqnarray}
If we replace $\mathcal{Y}$ by $e^{\mathcal{H} + \log \mathcal{A}}$, we then have
\begin{eqnarray}\label{eq:var formula}
\mathrm{Tr} e^{\mathcal{H} + \log \mathcal{A}} = \max\limits_{\mathcal{X} \succ \mathcal{O}} \Big\{ \mathrm{Tr}\mathcal{X} \star \mathcal{H} - \mathbb{D}(\mathcal{X} \parallel \mathcal{A})  + \mathrm{Tr}\mathcal{X} \Big\}
\end{eqnarray}
where $\mathbb{D}(\mathcal{X} \parallel  \mathcal{A})$ is the quantum relative entropy between two tensor operators. For real number $t \in [0, 1]$ and two positive-definite tensors $\mathcal{A}_1, \mathcal{A}_2$, we have 
\begin{eqnarray}
\mathrm{Tr} e^{\mathcal{H} + \log (t \mathcal{A}_1 + (1-t) \mathcal{A}_2 )} &=& 
\max_{\mathcal{X} \succ \mathcal{O}} \Big\{ \mathrm{Tr} \mathcal{X}\mathcal{H} 
- \mathbb{D}(\mathcal{X} \parallel t \mathcal{A}_1 + (1-t) \mathcal{A}_2) + \mathrm{Tr}\mathcal{X} \Big\}  \nonumber \\
&\geq & t \max_{\mathcal{X} \succ \mathcal{O}} \Big\{ \mathrm{Tr} \mathcal{X}\mathcal{H} 
- \mathbb{D}(\mathcal{X} \parallel  t \mathcal{A}_1) 
+ \mathrm{Tr}\mathcal{X} \Big\} \nonumber \\ 
& &+ (1-t) \max_{\mathcal{X} \succ \mathcal{O}} \Big\{ \mathrm{Tr} \mathcal{X}\mathcal{H} 
- \mathbb{D}(\mathcal{X}  \parallel  (1-t) \mathcal{A}_2) + \mathrm{Tr}\mathcal{X} \Big\} \nonumber \\
& = & t \mathrm{Tr}e^{\mathcal{H} + \log \mathcal{A}_1} +  (1-t) \mathrm{Tr}e^{\mathcal{H} + \log \mathcal{A}_2},
\end{eqnarray}
where the first and last equalities are obtained based on the variational formula provided by Eq.~\eqref{eq:var formula}, and the inequality is due to the joint convexity property of the 
relative entropy from Leamm~\ref{lma:joint conv of rela entropy}.
\end{proof}

Based on Lieb's concavity theorem for tensors, we have the following corollary.
\begin{corollary}\label{cor:3.3}
Let $\mathcal{A}$ be a fixed Hermitian tensor, and let $\mathcal{X}$ be a random Hermitian tensor, then we have
\begin{eqnarray}
\mathbb{E} \mathrm{Tr} e^{\mathcal{A} + \mathcal{X}} \leq \mathrm{Tr} e^{\mathcal{A} + \log \left( \mathbb{E} e^{\mathcal{X}} \right) }.
\end{eqnarray}
\end{corollary}
\begin{proof}
Define the random tensor $\mathcal{Y} = e^{\mathcal{X}}$, we have
\begin{eqnarray}
\mathbb{E} \mathrm{Tr} e^{\mathcal{A} + \mathcal{X}} = 
\mathbb{E} \mathrm{Tr} e^{\mathcal{A} + \log \mathcal{Y} } 
\leq \mathrm{Tr} e^{\mathcal{A} + \log\left( \mathbb{E} \mathcal{Y} \right) } 
= \mathrm{Tr} e^{\mathcal{A} + \log \left( \mathbb{E} e^{\mathcal{X}} \right) },
\end{eqnarray}
where the inequality is based on Lieb's concavity theorem for tensors~\ref{thm:Lieb concavity thm} and Jensen's inequality .
\end{proof}

\subsection{Tail Bounds for Independent Sums}\label{sec:Tail Bounds for Independent Sums}

This section will present the tail bound for the sum of independent random tensors and several corollaries according to this tail bound for independent sums. We begin with the subadditivity lemma of tensor cumulant-generating functions.

\begin{lemma}\label{lma:subadditivity of tensor cgfs}
Given a finite sequence of independent Hermitian random tensors $\{ \mathcal{X}_i \}$, we have 
\begin{eqnarray}\label{eq1:lma:subadditivity of tensor cgfs}
\mathbb{E} \mathrm{Tr} \exp\left( \sum\limits_{i=1}^n t \mathcal{X}_i \right) \leq 
\mathrm{Tr} \exp\left( \sum\limits_i^n \log \mathbb{E} e^{t \mathcal{X}_i} \right),~~\mbox{for $t \in \mathbb{R}$.}  
\end{eqnarray}
\end{lemma}
\begin{proof}
We first define the following term for the tensor cumulant-generating function for $\mathcal{X}_i$ as:
\begin{eqnarray}
\mathbb{K}_{i}(t) &\define& \log (\mathbb{E} e^{t \mathcal{X}_i}).
\end{eqnarray}
Then, we define the Hermitian tensor $\mathcal{H}_k$ as 
\begin{eqnarray}\label{eq2:lma:subadditivity of tensor cgfs}
\mathcal{H}_k(t) = \sum\limits_{i = 1}^{k-1} t \mathcal{X}_k + \sum\limits_{i = k+1}^n \mathbb{K}_{i}(t).
\end{eqnarray}

By applying Eq.~\eqref{eq2:lma:subadditivity of tensor cgfs} to Theorem~\ref{thm:Lieb concavity thm} repeatedly for $k = 1,2,\cdots,n$, we have 
\begin{eqnarray}
\mathbb{E} \mathrm{Tr} \exp\left( \sum\limits_{i=1}^n t \mathcal{X}_i \right) 
&=_1& \mathbb{E}_0 \cdots \mathbb{E}_{n-1} \mathrm{Tr} \exp\left(\sum\limits_{i=1}^{n-1} t\mathcal{X}_i + t\mathcal{X}_n \right) \nonumber \\
&\leq&  \mathbb{E}_0 \cdots \mathbb{E}_{n-2} \mathrm{Tr} \exp\left( \sum\limits_{i=1}^{n-1}t \mathcal{X}_i + \log\left( \mathbb{E}_{n - 1}e^{t\mathcal{X}_n} \right) \right) \nonumber \\
& = &  \mathbb{E}_0 \cdots \mathbb{E}_{n-2} \mathrm{Tr} \exp\left( \sum\limits_{i=1}^{n-2} t \mathcal{X}_i + t \mathcal{X}_{n-1} +\mathbb{K}_{n}(t).  \right) \nonumber \\
& \leq &  \mathbb{E}_0 \cdots \mathbb{E}_{n-3} \mathrm{Tr} \exp\left( \sum\limits_{i=1}^{n-2} t \mathcal{X}_i +  \mathbb{K}_{n-1}(t) + \mathbb{K}_{n}(t)  \right) \nonumber \\
\cdots 
& \leq & \mathrm{Tr} \exp\left( \sum\limits_{i = 1}^{n}  \mathbb{K}_{i}(t)  \right)
\end{eqnarray}
where the equality $=_1$ is based on the law of total expectation by defining $\mathbb{E}_i$ as the conditional expectation given $\mathcal{X}_1, \cdots, \mathcal{X}_i$.
\end{proof}

We are ready to present the theorem for the tail bound of independent sums. 
\begin{theorem}[Master Tail Bound for Independent Sum of Random Tensors]\label{thm:Master Tail Bound for Independent Sum of Random Tensors}
Given a finite sequence of independent Hermitian random tensors $\{ \mathcal{X}_i \}$, we have 
\begin{eqnarray}\label{eq1:thm:master tail bound}
\mathrm{Pr} \left( \lambda_{\max} (\sum\limits_{i=1}^n \mathcal{X}_i) \geq \theta \right)
& \leq & \inf\limits_{t > 0} \Big\{ e^{- t \theta} \mathrm{Tr}\exp \left( \sum\limits_{i=1}^{n} \log \mathbb{E} e^{t \mathcal{X}_i}  \right) \Big\}. 
\end{eqnarray}
\end{theorem}
\begin{proof}
By substituting the Lemma~\ref{lma:subadditivity of tensor cgfs} into the Laplace transform bound provided by the Lemma~\ref{lma: Laplace Transform Method}, this theorem is established. 
\end{proof}

Several useful corollaries will be provided based on Theorem~\ref{thm:Master Tail Bound for Independent Sum of Random Tensors}.

\begin{corollary}\label{cor:3_7}
Given a finite sequence of independent Hermitian random tensors $\{ \mathcal{X}_i \}$ with dimensions in $\mathbb{C}^{I_1 \times \cdots \times I_M \times I_1 \times \cdots \times I_M}$. If there is a function $f: (0, \infty) \rightarrow [0, \infty]$ and a sequence of non-random Hermitian tensors $\{ \mathcal{A}_i \}$ with following condition:
\begin{eqnarray}\label{eq:cond in cor 3_7}
f(t) \mathcal{A}_i \succeq  \log \mathbb{E} e^{t \mathcal{X}_i},~~\mbox{for $t > 0$.}
\end{eqnarray}
Then, for all $\theta \in \mathbb{R}$, we have
\begin{eqnarray}
\mathrm{Pr} \left( \lambda_{\max}\left(\sum\limits_{i=1}^n \mathcal{X}_i \right) \geq \theta \right)
& \leq & \mathbb{I}_{1}^M \inf\limits_{t > 0}\Big\{\exp\left[ - t \theta + f(\theta)\lambda_{\max}\left(\sum\limits_{i=1}^n \mathcal{A}_i \right) \right] \Big\}
\end{eqnarray}
\end{corollary}
\begin{proof}
From the condition provided by Eq.~\eqref{eq:cond in cor 3_7} and Theorem~\ref{thm:Master Tail Bound for Independent Sum of Random Tensors}, we have 
\begin{eqnarray}
\mathrm{Pr} \left( \lambda_{\max}\left(\sum\limits_{i=1}^n \mathcal{X}_i \right) \geq \theta \right) & \leq & e^{- t \theta} \mathrm{Tr} \exp (f(\theta) \sum\limits_{i=1}^n \mathcal{A}_i) \nonumber \\
& \leq &  (I_1 \cdots I_M) e^{- t \theta} \lambda_{\max} \left( \exp (f(\theta) \sum\limits_{i=1}^n \mathcal{A}_i) \right) \nonumber \\
& = &  \mathbb{I}_1^M  e^{- t \theta} \exp\left( f(\theta) \lambda_{\max} \left(\sum\limits_{i=1}^n \mathcal{A}_i \right) \right),
\end{eqnarray}
where the second inequality holds since we bound the trace of a positive-definite tensor by the dimension size $I_1 \cdots I_M$ (the multiplication of $M$ positive integers) times the maximum eigenvalue; the last equality is based on the spectral mapping theorem since the function $f$ is nonnegative. This theorem is proved by taking the infimum over positive $t$. 
\end{proof}

\begin{corollary}\label{cor:3_9}
Given a finite sequence of independent Hermitian random tensors $\{ \mathcal{X}_i \}$ with dimensions in $\mathbb{C}^{I_1 \times \cdots \times I_M \times I_1 \times \cdots \times I_M}$. For all $\theta \in \mathbb{R}$, we have
\begin{eqnarray}
\mathrm{Pr} \left( \lambda_{\max}\left(\sum\limits_{i=1}^n \mathcal{X}_i \right) \geq \theta \right)
& \leq & \mathbb{I}_1^M \inf\limits_{t > 0}  \Big\{  \exp\left[ - t \theta +n \log \lambda_{\max} \left( \frac{ \sum\limits_{i=1}^{n} \mathbb{E} e^{t \mathcal{X}_i}}{n}\right)  \right]  \Big\} \nonumber \\
\end{eqnarray}
\end{corollary}
\begin{proof}
From tensor logarithm given by Eq.~\eqref{eq:log concave}, we have 
\begin{eqnarray}
\sum\limits_{i=1}^{n} \log \mathbb{E} e^{t \mathcal{X}_i} = n \cdot \frac{1}{n}  \sum\limits_{i=1}^{n} \log \mathbb{E} e^{t \mathcal{X}_i} \preceq n \log \left( \frac{1}{n} \sum\limits_{i=1}^{n} \mathbb{E} e^{t \mathcal{X}_i} \right),
\end{eqnarray}
and from the trace exponential monotone property provided by Eq.~\eqref{eq:trace exp monotone}, we have 
\begin{eqnarray}
\mathrm{Pr} \left( \lambda_{\max}\left(\sum\limits_{i=1}^n \mathcal{X}_i \right) \geq \theta \right)
 \leq  e^{- t \theta} \mathrm{Tr} \exp \left(n \log \left( \frac{1}{n} \sum\limits_{i=1}^n \mathbb{E} e^{ t \mathcal{X}_i }  \right) \right)~~~~~~~~~~~~~~~~~~ \nonumber \\
\leq  ( I_1 \cdots I_M) \inf\limits_{t > 0}  \Big\{ \exp\left[ - t \theta +n \log \lambda_{\max} \left( \frac{ \sum\limits_{i=1}^{n} \mathbb{E} e^{t \mathcal{X}_i}}{n}\right)  \right] \Big\},
\end{eqnarray}
where the last inequality holds since we bound the trace of a positive-definite tensor by the dimension size $I_1 \cdots I_M$ (the multiplication of $M$ positive integers) times the maximum eigenvalue and apply spectral mapping theorem.
\end{proof}

\section{Tensor with Random Series}\label{sec:Tensor Gaussian Series}

A tensor Gaussian series is one of the simplest cases of a sum of independent random tensors. For scalers, a Gaussian series with real coefficients satisfies a normal-type tail bound where the variance is controlled by the sum of squares coefficients. The main purpose of the first Section~\ref{sec:Tensor with Gaussian and Rademacher Random Series} is to extend this scenario to tensors. In Section~\ref{sec:A Gaussian Tensor with Nonuniform Variances}, we will apply results from Section~\ref{sec:Tensor with Gaussian and Rademacher Random Series} to consider Gaussian tensor with nonuniform variances. Finally, we will provide the lower and upper bounds of tensor expectation in Section~\ref{sec:Lower and Upper Bounds of Tensor Expectation}.

\subsection{Tensor with Gaussian and Rademacher Random Series}\label{sec:Tensor with Gaussian and Rademacher Random Series}

We begin with a lemma about moment-generating functions of Rademacher and Gaussian normal random variables. 

\begin{lemma}\label{lma:mgf of Rademacher and normal rvs.}
Suppose that the tensor $\mathcal{A}$ is Hermitian.  Given a Gaussian normal random variable $\alpha$ and a Rademacher random variable $\beta$, then, we have
\begin{eqnarray}
\mathbb{E}e^{\alpha t \mathcal{A}} =  e^{t^2 \mathcal{A}^2 / 2}  \mbox{~~and~~} e^{t^2 \mathcal{A}^2 / 2} \succeq \mathbb{E}e^{\beta t \mathcal{A}},
\end{eqnarray}   
where $t \in \mathbb{R}$.
\end{lemma}
\begin{proof}
For the Gaussian normal random variable, because we have
\begin{eqnarray}
\mathbb{E} (\alpha^{2n}) = \frac{ (2i)!}{i ! 2^i}  \mbox{~~and~~} \mathbb{E} (\alpha^{2i + 1}) = 0, 
\end{eqnarray}
where $i = 0,1,2,\cdots$; then
\begin{eqnarray}
\mathbb{E}e^{\alpha t \mathcal{A}} &=& \mathcal{I} + \sum\limits_{i=1}^{\infty} \frac{  \mathbb{E}(\alpha^{2i}) (t \mathcal{A})^{2i}}{ (2i)! } \nonumber \\
&=& \mathcal{I} + \sum\limits_{i=1}^{\infty} \frac{ (t^2 \mathcal{A}^2/2)^i}{ i! } = e^{t^2 \mathcal{A}^2 /2}. 
\end{eqnarray}

For the Rademacher random variable, we have 
\begin{eqnarray}
\mathbb{E} e^{\beta  t \mathcal{A}} = \cosh (t \mathcal{A}) \preceq e^{t^2 \mathcal{A}^2/2}.
\end{eqnarray}
Therefore, this Lemma is proved.
\end{proof}

We are ready to present the main Theorem of this section about Hermitian tensors with Gaussian and Rademacher series.

\begin{theorem}[Hermitian Tensor with Gaussian and Rademacher Series]\label{thm:TensorGaussianNormalSeries}
Given a finite sequence $\mathcal{A}_i$ of fixed Hermitian tensors with dimensions as $\mathbb{C}^{I_1 \times \cdots \times I_M \times I_1 \times \cdots \times I_M}$, and let $\{ \alpha_i \}$ be a finite sequence of independent normal variables. We define  
\begin{eqnarray}\label{eq:4_5}
\sigma^2 &\define& \left\Vert \sum\limits_i^{n} \mathcal{A}^2_i \right\Vert,
\end{eqnarray}
then, for all $\theta \geq 0$, we have 
\begin{eqnarray}\label{eq:4_3}
\mathrm{Pr}\left( \lambda_{\max} \left( \sum\limits_{i=1}^{n}\alpha_i \mathcal{A}_i \right) \geq \theta \right) \leq \mathbb{I}_1^M e^{-\frac{\theta^2 }{2 \sigma^2}},
\end{eqnarray}
and
\begin{eqnarray}\label{eq:4_4}
\mathrm{Pr}\left( \left\Vert \sum\limits_{i=1}^{n}\alpha_i \mathcal{A}_i \right\Vert \geq \theta \right) \leq 2 \mathbb{I}_1^M e^{-\frac{\theta^2 }{2 \sigma^2}}.
\end{eqnarray}
This theorem is also valid for a finite sequence of independent Rademacher random variables $\{ \alpha_i \}$.
\end{theorem}
\begin{proof}
Given a finite sequence of independent Gaussian or Rademacher random variables $\{ \alpha_i \}$, from Lemma~\ref{lma:mgf of Rademacher and normal rvs.}, we have 
\begin{eqnarray}
e^{\frac{t^2 \mathcal{A}_i^2}{2}} \succeq \mathbb{E}e^{\alpha_i t \mathcal{A}_i}.
\end{eqnarray}
From the definition in Eq.~\eqref{eq:4_5} and Corollary~\ref{cor:3_7}, we have 
\begin{eqnarray}\label{eq:4_9}
\mathrm{Pr}\left( \lambda_{\max} \left( \sum\limits_{i=1}^{n}\alpha_i \mathcal{A}_i \right) \geq \theta \right) \leq  \mathbb{I}_1^M \inf\limits_{t > 0}\Big\{ e^{- t \theta + \frac{t^2 \sigma^2}{2} } \Big\} = \mathbb{I}_1^M e^{- \frac{\theta^2}{2 \sigma^2}}.
\end{eqnarray}
This establishes  Eq.~\eqref{eq:4_3}. For  Eq.~\eqref{eq:4_4}, we have to apply following facts: $\left\Vert \mathcal{X} \right\Vert$ for any given Hermitian tensor $\mathcal{X}$. Because Gaussian and Rademacher random variables are symmetric, we have 
\begin{eqnarray}
\mathrm{Pr}\left( \lambda_{\max} \left( \sum\limits_{i=1}^{n} (-\alpha_i) \mathcal{A}_i \right) \geq \theta \right) =
\mathrm{Pr}\left(  - \lambda_{\min} \left( \sum\limits_{i=1}^{n}\alpha_i \mathcal{A}_i \right) \geq \theta \right) \leq \mathbb{I}_1^M e^{- \frac{\theta^2}{2 \sigma^2}}.
\end{eqnarray}
Then, we obtain  Eq.~\eqref{eq:4_4} as follows: 
\begin{eqnarray}
\mathrm{Pr}\left( \left\Vert \sum\limits_{i=1}^{n}\alpha_i \mathcal{A}_i \right\Vert \geq \theta \right)
&=& 2 \mathrm{Pr}\left( \lambda_{\max} \left( \sum\limits_{i=1}^{n}\alpha_i \mathcal{A}_i \right) \geq \theta \right)
\nonumber \\  
& \leq &  2 \mathbb{I}_1^M e^{-\frac{\theta^2 }{2 \sigma^2}}.
\end{eqnarray}
\end{proof}

From the Hermitian dilation definition provided by Eq.~\eqref{eq:Hermitian dialation def}, we can extend Theorem~\ref{thm:TensorGaussianNormalSeries} from square Hermitian tensor to rectangular tensor by the following corollary.
\begin{corollary}[Rectangular Tensor with Gaussian and Rademacher Series]\label{cor:Rectangular Tensor with Gaussian and Rademacher Series}
Given a finite sequence $\mathcal{A}_i$ of fixed Hermitian tensors with dimensions as $\mathbb{C}^{I_1 \times \cdots \times I_M \times J_1 \times \cdots \times J_M}$, and let $\{ \alpha_i \}$ be a finite sequence of indepedent normal variables. We define 
\begin{eqnarray}\
\sigma^2 &\define& \max \Bigg\{ \left\Vert \sum\limits_{i=1}^n \mathcal{A}_i \star_M \mathcal{A}^H_i \right\Vert,  \left\Vert \sum\limits_{i=1}^n \mathcal{A}^H_i \star_M \mathcal{A}_i \right\Vert\Bigg\}.
\end{eqnarray}
then, for all $\theta \geq 0$, we have 
\begin{eqnarray}
\mathrm{Pr}\left( \left\Vert \sum\limits_{i=1}^n \alpha_i \mathcal{A}_i \right\Vert \geq \theta \right) \leq \prod\limits_{m=1}^M(I_m + J_m) e^{-\frac{\theta^2 }{2 \sigma^2}}.
\end{eqnarray}
This corollary is also valid for a finite sequence of independent Rademacher random variables $\{ \alpha_i \}$.
\end{corollary}
\begin{proof}
Let $\{ \alpha_i \}$ be a finite sequence of independent Gaussian or Rademacher random variables. Consider a finite sequence of random Hermitian tensors $\{ \alpha_i \mathbb{D}(\mathcal{A}_i) \}$ with dimensions $\mathbb{C}^{(I_1 + J_1) \times \cdots \times (I_M + J_M) \times (I_1 + J_1) \times \cdots \times (I_M + J_M)}$, from the spectral relation of a dilation tensor provided by Eq.~\eqref{eq:spectral identity}, we have 
\begin{eqnarray}\label{eq1:cor:Rectangular Tensor Gaussian with Normal Series}
\left\Vert \sum\limits_i^n \alpha_i \mathcal{A}_i \right\Vert &=& \lambda_{\max} \left( \mathbb{D} \left(  \sum\limits_{i=1}^n\alpha_i \mathcal{A}_i \right)\right) = \lambda_{\max} \left(\sum\limits_{i=1}^n \alpha_i  \mathbb{D} \left(\mathcal{A}_i \right)\right).
\end{eqnarray}
Due to the following singular value relation
\begin{eqnarray}\label{eq2:cor:Rectangular Tensor Gaussian with Normal Series}
\sigma^2 &=& \left \Vert  \sum\limits_i \mathbb{D}( \mathcal{A}_i)^2\right\Vert
=
\left\Vert
\begin{bmatrix}
\sum\limits_{i=1}^n \mathcal{A}_i \star_M \mathcal{A}_i^H & \mathcal{O}   \\
\mathcal{O} &\sum\limits_{i=1}^n \mathcal{A}^H_i \star_M \mathcal{A}_i \\
\end{bmatrix}\right\Vert \nonumber \\
&=& \max \Bigg\{ \left\Vert\sum\limits_{i=1}^n  \mathcal{A}_i \star_M \mathcal{A}^H_i \right\Vert,  \left\Vert \sum\limits_i^{n} \mathcal{A}^H_i \star_M \mathcal{A}_i \right\Vert\Bigg\}.
\end{eqnarray}
From Eqs.~\eqref{eq1:cor:Rectangular Tensor Gaussian with Normal Series},  ~\eqref{eq2:cor:Rectangular Tensor Gaussian with Normal Series} and Theorem~\ref{thm:TensorGaussianNormalSeries}, this corollary is proved.
\end{proof}

% 4.3
\subsection{A Gaussian Tensor with Nonuniform Variances}\label{sec:A Gaussian Tensor with Nonuniform Variances}

In this section, we will apply results obtained from the previous section to consider Gaussian tensor with nonuniform variances
\begin{corollary}\label{cor:Gaussian Tensor with Nonuniform Var}
Given a tensor $\mathcal{A} \in \mathbb{C}^{I_1 \times \cdots \times I_M \times J_1 \times \cdots \times J_M}$ and a random tensor $\mathcal{X} \in $ \\ $ \mathbb{C}^{I_1 \times \cdots \times I_M \times J_1 \times \cdots \times J_M}$ whose entries are independent standard Gaussian normal random variables. Let $\circ$ be used to represent a Hadamard product (entrywise) between two tensors with same dimensions. Then, we have
\begin{eqnarray}
\mathrm{Pr} \left( \left\Vert \mathcal{X} \circ \mathcal{A} \right\Vert \geq \theta \right) 
\leq \prod\limits_{m=1}^M (I_m + J_m) e^{-\frac{\theta^2 }{2 \sigma^2}}, 
\end{eqnarray}
where 
\begin{eqnarray}
\sigma^2 = \max \Bigg\{\max_{i_1,\cdots,i_M} \left\Vert a_{i_1,\cdots,i_M,:}\right\Vert^2, 
\max_{j_1,\cdots,j_M} \left\Vert a_{:,j_1,\cdots,j_M}\right\Vert^2 \Bigg\},
\end{eqnarray}
where $a_{i_1,\cdots,i_M,:}$ and $a_{:,j_1,\cdots,j_M}$ represent the row-part of the tensor $\mathcal{A}$ and column-part of the tensor $\mathcal{A}$, respectively.
\end{corollary}
\begin{proof}
Since we can decompose the tensor $\mathcal{X} \circ \mathcal{A}$ as:
\begin{eqnarray}
\mathcal{X} \circ \mathcal{A}  &=& \sum\limits_{i_1,\cdots,i_M,j_1,\cdots,j_M} x_{i_1,\cdots,i_M,j_1,\cdots,j_M} a_{i_1,\cdots,i_M,j_1,\cdots,j_M} \mathcal{E}_{i_1,\cdots,i_M,j_1,\cdots,j_M},  
\end{eqnarray}
where $\mathcal{E}_{i_1,\cdots,i_M,j_1,\cdots,j_M} \in \mathbb{C}^{I_1 \times \cdots \times I_M \times J_1 \times \cdots \times J_M}$ is the tensor with all zero entries except unity at the position $i_1,\cdots,i_M,j_1,\cdots,j_M$; then, we have
\begin{eqnarray}
\sum\limits_{i_1,\cdots,i_M,j_1,\cdots,j_M} \left(a_{i_1,\cdots,i_M,j_1,\cdots,j_M} \mathcal{E}_{i_1,\cdots,i_M,j_1,\cdots,j_M} \right) \left( \overline{a_{i_1,\cdots,i_M,j_1,\cdots,j_M} \mathcal{E}_{i_1,\cdots,i_M,j_1,\cdots,j_M} } \right)  \nonumber \\
= \sum\limits_{i_1,\cdots,i_M} \left( \sum\limits_{j_1,\cdots,j_M} \left\| a_{i_1,\cdots,i_M,j_1,\cdots,j_M}\right\|^2 \right)    \mathcal{E}_{i_1,\cdots,i_M,i_1,\cdots,i_M} ~~~~~~~~~~~~~~~~~~~~~~~~~~~~~~ \nonumber \\
= \mbox{diag} \left(\left\Vert a_{1,\cdots,1:}\right\Vert^2, \left\Vert a_{1,\cdots,2:}\right\Vert^2, \cdots, \left\Vert a_{I_1,\cdots,I_M:}\right\Vert^2\right),~~~~~~~~~~~~~~~~~~~~~~~~~~~~~~~~~~~~
\end{eqnarray}
and, similarly, 
\begin{eqnarray}
\sum\limits_{i_1,\cdots,i_M,j_1,\cdots,j_M} \left( \overline{a_{i_1,\cdots,i_M,j_1,\cdots,j_M} \mathcal{E}_{i_1,\cdots,i_M,j_1,\cdots,j_M}} \right) \left(a_{i_1,\cdots,i_M,j_1,\cdots,j_M} \mathcal{E}_{i_1,\cdots,i_M,j_1,\cdots,j_M} \right)   \nonumber \\ 
= \sum\limits_{j_1,\cdots,j_M} \left( \sum\limits_{i_1,\cdots,i_M} \left\| a_{i_1,\cdots,i_M,j_1,\cdots,j_M}\right\|^2 \right)    \mathcal{E}_{j_1,\cdots,j_M,j_1,\cdots,j_M} ~~~~~~~~~~~~~~~~~~~~~~~~~~~~~~ \nonumber \\
= \mbox{diag} \left(\left\Vert a_{:1,\cdots,1}\right\Vert^2, \left\Vert a_{:,1,\cdots,2}\right\Vert^2, \cdots, \left\Vert a_{:,J_1,\cdots,J_M}\right\Vert^2\right).~~~~~~~~~~~~~~~~~~~~~~~~~~~~~~~~~~~~
\end{eqnarray}
Therefore, we have 
\begin{eqnarray}
\sigma^2 &=& \max\Bigg\{\mbox{diag} \left(\left\Vert a_{1,\cdots,1:}\right\Vert^2, \left\Vert a_{1,\cdots,2:}\right\Vert^2, \cdots, \left\Vert a_{I_1,\cdots,I_M:}\right\Vert^2\right), \nonumber \\
&  & \mbox{diag} \left(\left\Vert a_{:1,\cdots,1}\right\Vert^2, \left\Vert a_{:,1,\cdots,2}\right\Vert^2, \cdots, \left\Vert a_{:,J_1,\cdots,J_M}\right\Vert^2\right) \Bigg\} \nonumber \\
&=& \max \Bigg\{\max_{i_1,\cdots,i_M} \left\Vert a_{i_1,\cdots,i_M,:}\right\Vert^2, 
\max_{j_1,\cdots,j_M} \left\Vert a_{:,j_1,\cdots,j_M}\right\Vert^2 \Bigg\}.
\end{eqnarray}
Finally, from Corollary~\ref{cor:Rectangular Tensor with Gaussian and Rademacher Series}, this Corollary is proved.
\end{proof}

%4.4
\subsection{Lower and Upper Bounds of Tensor Expectation}\label{sec:Lower and Upper Bounds of Tensor Expectation}
Given a finite sequence $\mathcal{A}_i$ of fixed Hermitian tensors with dimensions as $\mathbb{C}^{I_1 \times \cdots \times I_M \times I_1 \times \cdots \times I_M}$, and let $\{ \alpha_i \}$ be a finite sequence of indepedent normal variables. We define following random tensor 
\begin{eqnarray}
\mathcal{X} = \sum\limits_{i=1}^n \alpha_i \mathcal{A}_i.
\end{eqnarray}
From Theorem~\ref{thm:TensorGaussianNormalSeries}, we have 
\begin{eqnarray}\label{eq:upper bound for exp tensor}
\mathbb{E}\left( \left\Vert \mathcal{X} \right\Vert^2 \right)& = & \int_0^{\infty} \mathrm{Pr}\left(
\left\Vert \mathcal{X} \right\Vert > \sqrt{t} \right) dt \leq 2 \sigma^2 \log(2 \mathbb{I}_1^M)  \nonumber \\
&  &
+ 2 (I_1\cdots I_M) \int_{2 \sigma^2 \log(2\mathbb{I}_1^M)}^{\infty} = 2 \sigma^2 \log (2 e \mathbb{I}_1^M).
\end{eqnarray}
On the other hand, from Jesen's inequality, we have 
\begin{eqnarray}\label{eq:lower bound for exp tensor}
\mathbb{E}\left( \left\Vert \mathcal{X} \right\Vert^2 \right) = \mathbb{E} \left\Vert \mathcal{X}^2 \right\Vert  \geq \left\Vert \mathbb{E}(\mathcal{X}^2) \right\Vert =  \left\Vert \sum\limits_{i=1}^n \mathcal{A}^2_i\right\Vert = \sigma^2.
\end{eqnarray}
From both Eqs.~\eqref{eq:upper bound for exp tensor} and~\eqref{eq:lower bound for exp tensor},
we have following relation:
\begin{eqnarray}
c \sigma \leq \mathbb{E} \left\Vert \mathcal{X} \right\Vert \leq \sigma \sqrt{2 \log (2 e \mathbb{I}_1^M)}.
\end{eqnarray}
This shows that the tensor variance parameter $\sigma^2$ controls the expected norm $ \mathbb{E} \left\Vert \mathcal{X} \right\Vert$ with square root of logarithmic function for the tensor dimensions. 

\section{Tensor Chernoff Bounds}\label{sec:Tensor Chernoff Bounds}

The traditional Chernoff bounds concern the sum of independent, nonnegative, and uniformly bounded random variables. In this work, we will try to extend such Chernoff bounds under the scenario of random tensors.

\subsection{Tensor Chernoff Bounds Derivations}\label{sec:Tensor Chernoff Bounds Derivations}

We begin to present a Lemma about the semidefinite relation for the tensor moment-generating function of a random positive semidefinite contraction.

\begin{lemma}\label{lma:Chernoff MGF}
Given a random positive semifefinite tensor with $\lambda_{\max}(\mathcal{X})\leq 1$, then, for any $t \in \mathbb{R}$, we have 
\begin{eqnarray}
 \mathcal{I}+ (e^t - 1) \mathbb{E} \mathcal{X} \succeq \mathbb{E} e^{t \mathcal{X}}.
\end{eqnarray}
\end{lemma}
\begin{proof}
Consider a convex function $g(x) = e^{t x}$, we have  
\begin{eqnarray}
1 + (e^t - 1)x \geq g(x),
\end{eqnarray}
where $x \in [0,1]$. Since the eigenvalues of the random tensor $\mathcal{X}$ lie in the interval $[0, 1]$, from Eq.~\eqref{eq:tensor psd ordering}, we obtain
\begin{eqnarray}
\mathcal{I} + (e^t - 1)\mathcal{X} \succeq e^{t \mathcal{X}}.
\end{eqnarray}
Then, this Lemma is proved by taking the expectation with respect to the random tensor $\mathcal{X}$.
\end{proof}

Given two real values $a, b \in [0, 1]$, we define \emph{binary information divergence} of $a$ and $b$, expressed by $\mathfrak{D}(a || b)$, as 
\begin{eqnarray}\label{eq:def of binary info div}
\mathfrak{D}(a || b) \define a \log \frac{a}{b} + (1-a) \frac{1-a}{1-b}.
\end{eqnarray}    

We are ready to present tensor Chernoff inequality.
\begin{theorem}[Tensor Chernoff Bound I]\label{thm:TensorChernoffBoundI}
Consider a sequence $\{ \mathcal{X}_i  \in \mathbb{C}^{I_1 \times \cdots \times I_M  \times I_1 \times \cdots \times I_M } \}$ of independent, random, Hermitian tensors that satisfy
\begin{eqnarray}
\mathcal{X}_i \succeq \mathcal{O} \mbox{~~and~~} \lambda_{\max}(\mathcal{X}_i) \leq 1 
\mbox{~~ almost surely.}
\end{eqnarray}
Define following two quantaties:
\begin{eqnarray}
\overline{\mu}_{\max} \define \lambda_{\max}\left( \frac{1}{n} \sum\limits_{i=1}^{n} \mathbb{E} \mathcal{X}_i \right) \mbox{~~and~~} 
\overline{\mu}_{\min} \define \lambda_{\min}\left( \frac{1}{n} \sum\limits_{i=1}^{n} \mathbb{E} \mathcal{X}_i \right),
\end{eqnarray}
then, we have following two inequalities:
\begin{eqnarray}\label{eq:Chernoff I Upper Bound}
\mathrm{Pr} \left( \lambda_{\max}\left( \frac{1}{n}\sum\limits_{i=1}^{n} \mathcal{X}_i \right) \geq \theta \right) \leq \mathbb{I}_1^M e^{- n \mathfrak{D}(\theta || \overline{\mu}_{\max} )},\mbox{~~ for $\overline{\mu}_{\max} \leq \theta \leq 1$;}
\end{eqnarray}
and
\begin{eqnarray}\label{eq:Chernoff I Lower Bound}
\mathrm{Pr} \left( \lambda_{\min}\left( \frac{1}{n}\sum\limits_{i=1}^{n} \mathcal{X}_i \right) \leq \theta \right) \leq \mathbb{I}_1^M e^{- n \mathfrak{D}(\theta || \overline{\mu}_{\min} )},\mbox{~~ for $0 \leq \theta \leq \overline{\mu}_{\min}$.}
\end{eqnarray}
\end{theorem}
\begin{proof}
From Lemma~\ref{lma:Chernoff MGF}, we have 
\begin{eqnarray}
\mathcal{I} + f(t) \mathbb{E} \mathcal{X}_i \succeq \mathbb{E} e^{t \mathcal{X}_i},
\end{eqnarray}
where $f(t) \define e^t -1$ for $t > 0$.
By applying Corollary~\ref{cor:3_9}, we obtain
\begin{eqnarray}\label{eq1:Chernoff I Upper Bound proof}
\mathrm{Pr} \left( \lambda_{\max}\left( \sum\limits_{i=1}^{n} \mathcal{X}_i \right) \geq \alpha \right) &\leq& \mathbb{I}_1^M \exp\left( - t \alpha + n \log\lambda_{\max}\left( \frac{1}{n} \sum\limits_{i=1}^{n} \left( \mathcal{I} + f(t) \mathbb{E} \mathcal{X}_i \right) \right) \right) \nonumber \\
&= &  \mathbb{I}_1^M \exp\left( - t \alpha+ n \log\lambda_{\max}\left(  \mathcal{I} + f(t)  \frac{1}{n} \sum\limits_{i=1}^{n} \mathbb{E} \mathcal{X}_i \right) \right) \nonumber \\ 
&=& \mathbb{I}_1^M \exp\left( - t \alpha + n \log \left(1 + f(t) \overline{\mu}_{\max} \right) \right) .
\end{eqnarray}
The last equality follows from the definition of $\overline{\mu}_{\max}$ and the eigenvalue map properties. When the value $t$ at the right-hand side of  Eq.~\eqref{eq1:Chernoff I Upper Bound proof} is 
\begin{eqnarray}\label{eq2:Chernoff I Upper Bound proof}
t = \log \frac{\alpha}{1 - \alpha} - \log \frac{\overline{\mu}_{\max}}{1 - \overline{\mu}_{\max}},
\end{eqnarray}
we can achieve the tightest upper bound at Eq.~\eqref{eq1:Chernoff I Upper Bound proof}. By substituting the value $t$ in Eq.~\eqref{eq2:Chernoff I Upper Bound proof} into Eq.~\eqref{eq1:Chernoff I Upper Bound proof} and change the variable $\alpha \rightarrow n \theta$, Eq.~\eqref{eq:Chernoff I Upper Bound} is proved. The next goal is to prove Eq.~\eqref{eq:Chernoff I Lower Bound}.

If we apply Lemma~\ref{lma:Chernoff MGF} to the sequence $\{- \mathcal{X}_i \}$, we have 
\begin{eqnarray}
\mathcal{I} - g(t) \mathbb{E} \mathcal{X}_i \succeq \mathbb{E} e^{t (-\mathcal{X}_i)},
\end{eqnarray}
where $g(t) \define 1 - e^t$ for $t > 0$.
By applying Corollary~\ref{cor:3_9} again, we obtain
\begin{eqnarray}\label{eq1:Chernoff I Lower Bound proof}
\mathrm{Pr} \left( \lambda_{\min}\left( \sum\limits_{i=1}^{n} \mathcal{X}_i \right) \leq \alpha \right) &=& \mathrm{Pr} \left( \lambda_{\max}\left( \sum\limits_{i=1}^{n} \left( - \mathcal{X}_i \right)\right) \geq \alpha \right) 
\nonumber \\
&\leq& \mathbb{I}_1^M \exp\left( t \alpha + n \log\lambda_{\max}\left( \frac{1}{n} \sum\limits_{i=1}^{n} \left( \mathcal{I} - g(t) \mathbb{E} \mathcal{X}_i \right) \right) \right) \nonumber \\
&=_1 &  \mathbb{I}_1^M \exp\left( t \alpha + n \log \left( 1 -f(t)  \lambda_{\min}\left(    \frac{1}{n} \sum\limits_{i=1}^{n} \mathbb{E} \mathcal{X}_i \right) \right) \right) \nonumber \\ 
&=& \mathbb{I}_1^M \exp\left( t \alpha + n \log \left(1 - g(t) \overline{\mu}_{\min} \right) \right),
\end{eqnarray}
where we apply the relation $\lambda_{\min}( -\frac{1}{n} \sum\limits_{i=1}^{n} \mathbb{E} \mathcal{X}_i  ) = -\lambda_{\max} (\frac{1}{n} \sum\limits_{i=1}^{n} \mathbb{E} \mathcal{X}_i )$ at the equality $=_1$. When the value $t$ at the right-hand side of  Eq.~\eqref{eq1:Chernoff I Lower Bound proof} is 
\begin{eqnarray}\label{eq2:Chernoff I Lower Bound proof}
t = log \frac{\overline{\mu}_{\max}}{1 - \overline{\mu}_{\max}} -  \log \frac{\alpha}{1 - \alpha},
\end{eqnarray}
% eq:Chernoff I Lower Bound
we can achieve the tightest upper bound at Eq.~\eqref{eq1:Chernoff I Lower Bound proof}. By substituting the value $t$ in Eq.~\eqref{eq2:Chernoff I Lower Bound proof} into Eq.~\eqref{eq1:Chernoff I Lower Bound proof} and change the variable $\alpha \rightarrow n \theta$, Eq.~\eqref{eq:Chernoff I Lower Bound} is proved also.
\end{proof}

The tensor Chernoff bounds discussed at Theorem~\ref{thm:TensorChernoffBoundI} is not related to $\mu_{\max}$ and $\mu_{\min}$ directly. Following theorem is another version of tensor Chernoff bounds to associate the probability range in terms of $\mu_{\max}$ and $\mu_{\min}$ directly and this format of tensor Chernoff bounds are easier to be applied. 

\begin{theorem}[Tensor Chernoff Bound II]\label{thm:TensorChernoffBoundII}
Consider a sequence $\{ \mathcal{X}_i  \in \mathbb{C}^{I_1 \times \cdots \times I_M  \times I_1 \times \cdots \times I_M } \}$ of independent, random, Hermitian tensors that satisfy
\begin{eqnarray}
\mathcal{X}_i \succeq \mathcal{O} \mbox{~~and~~} \lambda_{\max}(\mathcal{X}_i) \leq T
\mbox{~~ almost surely.}
\end{eqnarray}
Define following two quantaties:
\begin{eqnarray}
\mu_{\max} \define \lambda_{\max}\left( \sum\limits_{i=1}^{n} \mathbb{E} \mathcal{X}_i \right) \mbox{~~and~~} 
\mu_{\min} \define \lambda_{\min}\left( \sum\limits_{i=1}^{n} \mathbb{E} \mathcal{X}_i \right),
\end{eqnarray}
then, we have following two inequalities:
\begin{eqnarray}\label{eq:Chernoff II Upper Bound}
\mathrm{Pr} \left( \lambda_{\max}\left( \sum\limits_{i=1}^{n} \mathcal{X}_i \right) \geq (1+\theta) \mu_{\max} \right) \leq \mathbb{I}_1^M \left(\frac{e^{\theta}}{ (1 + \theta)^{1 + \theta}  }\right)^{\mu_{\max}/T} ,\mbox{~~ for $\theta \geq 0$;}
\end{eqnarray}
and
\begin{eqnarray}\label{eq:Chernoff II Lower Bound}
\mathrm{Pr} \left( \lambda_{\min}\left( \sum\limits_{i=1}^{n} \mathcal{X}_i \right) \leq (1 - \theta) \mu_{\min} \right) \leq \mathbb{I}_1^M \left(\frac{e^{-\theta}}{ (1 - \theta)^{1 - \theta}  }\right)^{\mu_{\min}/T} ,\mbox{~~ for $\theta \in [0,1]$.}
\end{eqnarray}
\end{theorem}
\begin{proof}
Without loss of generality, we can assume $T=1$ in our proof. From  Eq.~\eqref{eq1:Chernoff I Upper Bound proof} and the inequality $\log(1 + x) \leq x$ for $x > -1$, we have 
\begin{eqnarray}\label{eq1:Chernoff II Upper Bound}
\mathrm{Pr}\left( \lambda_{\max}\left( \sum\limits_{i=1}^{n} \mathcal{X}_i \right) \geq t \right) \leq \mathbb{I}_1^{M}\exp(- \delta t + f(\delta) \mu_{\max})
\end{eqnarray}
By selecting $\delta = \log (1 + \theta)$ and $t \rightarrow (1 + \theta) \mu_{\max}$, we can establish  Eq.~\eqref{eq:Chernoff II Upper Bound}.

From  Eq.~\eqref{eq1:Chernoff I Lower Bound proof} and the inequality $\log(1 + x) \leq x$ for $x > -1$, we have 
\begin{eqnarray}\label{eq1:Chernoff II Lower Bound}
\mathrm{Pr}\left( \lambda_{\min}\left( \sum\limits_{i=1}^{n} \mathcal{X}_i \right) \leq t \right) \leq \mathbb{I}_1^{M}\exp(- \delta t - f(\delta) \mu_{\min})
\end{eqnarray}
By selecting $\delta =- \log (1 - \theta)$ and $t \rightarrow (1 - \theta) \mu_{\min}$, we can establish  Eq.~\eqref{eq:Chernoff II Lower Bound}. Therefore, this theorem is proved. 
\end{proof}

\subsection{Application of Tensor Chernoff Bounds}\label{sec:Application of Tensor Chernoff Bounds}

% Remark 5.4

Consider a general random tensor $\mathcal{X} \in \mathbb{C}^{I_1 \times \cdots \times I_M \times J_1 \times \cdots \times J_N}$, we can express the $\mathcal{X}$ as 
\begin{eqnarray}
\mathcal{X} = [x_{1}, x_{2}, \cdots, x_{J_1\cdots J_N}], 
\end{eqnarray}
where $x_i$ is a family of independent random tensors in $ \mathbb{C}^{I_1 \times \cdots \times I_M}$ (vector part in $\mathcal{X}$. The squared norm of $\mathcal{X}$ can be expressed as 
\begin{eqnarray}\label{eq1:Remark 5_4}
\left\Vert \mathcal{X} \right\Vert^2 = \lambda_{\max}(\mathcal{X} \star_N \mathcal{X}^H) = \lambda_{\max}(\sum\limits_{i=1}^{J_1\cdots J_N} x_i \star_0 \overline{x_i}).
\end{eqnarray}
Similarly, for the minimum singular value of the tensor $\mathcal{X}$, we have 
\begin{eqnarray}\label{eq2:Remark 5_4}
\mbox{Minimum singular value of $\mathcal{X}$} = \lambda_{\min}(\mathcal{X} \star_N \mathcal{X}^H) = \lambda_{\min}(\sum\limits_{i=1}^{J_1\cdots J_N} x_i \star_0 x_i^H).
\end{eqnarray}
From tensor Chernoff bounds provided by Theorem~\ref{thm:TensorChernoffBoundII}, we can bound both Eqs.~\eqref{eq1:Remark 5_4} and~\eqref{eq2:Remark 5_4}. 

Another application of tensor Chernoff bounds is to estimate the expectation of the maximum eigenvalue of independent sum of random tensors. 
\begin{corollary}[Upper and Lower Bounds for the Maximum Eigenvalue]\label{cor:Bounds for the Maximum Eigenvalue}
Consider a sequence $\{ \mathcal{X}_i  \in \mathbb{C}^{I_1 \times \cdots \times I_M  \times I_1 \times \cdots \times I_M } \}$ of independent, random, Hermitian tensors that satisfy
\begin{eqnarray}
\mathcal{X}_i \succeq \mathcal{O} \mbox{~~and~~} \lambda_{\max}(\mathcal{X}_i) \leq T
\mbox{~~ almost surely.}
\end{eqnarray}
Then, we have
\begin{eqnarray}\label{eq1:cor:Bounds for the Maximum Eigenvalue}
\mu_{\max} \leq \mathbb{E} \lambda_{\max}\left( \sum\limits_{i=1}^n \mathcal{X}_i \right) \leq C \mathbb{I}_1^M e^{- \mu_{\max}/T},
\end{eqnarray}
where the constant value of $C$ is about 10.28.
\end{corollary}
\begin{proof}
The lower bound at Eq.~\eqref{eq1:cor:Bounds for the Maximum Eigenvalue} is true from the convexity of the function $\mathcal{A} \rightarrow \lambda_{\max}(\mathcal{A})$ and the Jensen's inequality. 

For the upper bound, we have 
\begin{eqnarray}\label{eq2:cor:Bounds for the Maximum Eigenvalue}
\mathbb{E} \lambda_{\max}\left( \sum\limits_{i=1}^n \mathcal{X}_i \right)
&=& \int_{0}^{\infty} \mathrm{Pr} \left( \lambda_{\max}\left( \sum\limits_{i=1}^n \mathcal{X}_i \right)  \geq t \right) d t  \nonumber \\
&\leq_1& \int_{0}^{\infty} \mathbb{I}_1^{M}\exp(- \delta t + (e^{\delta} - 1) \mu_{\max}/T)dt \nonumber \\
&=&  \frac{e^{e^{\delta}}}{\delta} \mathbb{I}_1^M e^{- \mu_{\max}/T} \nonumber \\
&\leq&  \frac{e^{e^{\delta_{opt}}}}{\delta_{opt}} \mathbb{I}_1^M e^{- \mu_{\max}/T}
 = C \mathbb{I}_1^M e^{- \mu_{\max}/T},
\end{eqnarray}
where the inequality $\leq_1$ comes from  Eq.~\eqref{eq1:Chernoff II Upper Bound} with the scaling factor $T$. If we select $\theta$ as the solution of the following relation $e^{\delta_{opt}} = \frac{1}{\delta_{opt}}$ to minimize the right-hand side of Eq.~\eqref{eq2:cor:Bounds for the Maximum Eigenvalue},  we have the desired upper bound when $\delta_{opt} \approx = 0.56699$. This corollary is proved. 
\end{proof}

% Bound of \mathbb{E} \lambda_{\max}\left( \sum\limits_i^{n} \mathcal{X}_i \right) 

\section{Tensor Bernstein Bounds}\label{sec:Tensor Bernstein Bounds}

For random variables, Bernstein inequalities give the upper tail of a sum of
independent, zero-mean random variables that are either bounded or subexponential. In this section, we wish to extend Bernstein bounds for a sum of zero-mean random tensors.

\subsection{Tensor Bernstein Bounds Derivation}\label{sec:Tensor Bernstein Bounds Derivation}

We will condier bounded Tensor Bernstein bounds first by considering the bounded Bernstein moment-generating function with the following Lemma.
\begin{lemma}\label{lam:Bounded Bernstein mgf}
Given a random Hermitian tensor $\mathcal{X} \in \mathbb{C}^{I_1 \times \cdots \times I_M  \times I_1 \times \cdots \times I_M }$ that satisfies:
\begin{eqnarray}\label{eq1:lam:Bounded Bernstein mgf}
\mathbb{E} \mathcal{X} = 0 \mbox{~~and~~} \lambda_{\max}(\mathcal{X} ) \leq 1 
\mbox{~~almost surely.} 
\end{eqnarray}
Then, we have
\begin{eqnarray}
 e^{  (e^t - t - 1) \mathbb{E}(\mathcal{X}^2)}   \succeq  \mathbb{E}e^{t \mathcal{X}}
\end{eqnarray}
where $t > 0$.
\end{lemma}
\begin{proof}
If we define a real function $g(x) \define \frac{e^{t x} - t x - 1}{x^2}$, it is easy to see that this function $g(x)$ is an increasing function for $0 < x \leq 1$. From Eq~\eqref{eq:tensor psd ordering}, we have 
\begin{eqnarray}\label{eq2:lam:Bounded Bernstein mgf}
g(\mathcal{X}) \preceq g(1) \mathcal{I}.
\end{eqnarray}
Moreover, we also have
\begin{eqnarray}\label{eq3:lam:Bounded Bernstein mgf}
e^{t \mathcal{X}} &=& \mathcal{I} + t \mathcal{X} + g(\mathcal{X}) \star_M \mathcal{X}^2 \nonumber \\
&\preceq&  \mathcal{I} + t \mathcal{X} + g(1) \mathcal{X}^2,
\end{eqnarray}
where the $\preceq$ comes from Eq.~\eqref{eq2:lam:Bounded Bernstein mgf}. By taking the expectation for both sides of Eq.~\eqref{eq3:lam:Bounded Bernstein mgf}, we then obtain
\begin{eqnarray}
\mathbb{E}e^{t \mathcal{X}} &\preceq& \mathcal{I} + g(1) \mathbb{E}\left( \mathcal{X}^2\right) \preceq e^{g(1)\mathbb{E}\left( \mathcal{X}^2\right) } \nonumber \\
&=&   e^{ (e^t - t - 1) \mathbb{E}(\mathcal{X}^2)}.
\end{eqnarray}
This lemma is established. 
\end{proof}

We are ready to present the Tensor Bernstein bounds for random tensors with bounded $\lambda_{\max}$. 

\begin{theorem}[Bounded $\lambda_{\max}$ Tensor Bernstein Bounds]\label{thm:Bounded Tensor Bernstein}
Given a finite sequence of independent Hermitian tensors $\{ \mathcal{X}_i  \in \mathbb{C}^{I_1 \times \cdots \times I_M  \times I_1 \times \cdots \times I_M } \}$ that satisfy
\begin{eqnarray}\label{eq1:thm:Bounded Tensor Bernstein}
\mathbb{E} \mathcal{X}_i = 0 \mbox{~~and~~} \lambda_{\max}(\mathcal{X}_i) \leq T 
\mbox{~~almost surely.} 
\end{eqnarray}
Define the total varaince $\sigma^2$ as: $\sigma^2 \define \left\Vert \sum\limits_i^n \mathbb{E} \left( \mathcal{X}^2_i \right) \right\Vert$.
Then, we have following inequalities:
\begin{eqnarray}\label{eq2:thm:Bounded Tensor Bernstein}
\mathrm{Pr} \left( \lambda_{\max}\left( \sum\limits_{i=1}^{n} \mathcal{X}_i \right)\geq \theta \right) \leq \mathbb{I}_1^M \exp \left( \frac{-\theta^2/2}{\sigma^2 + T\theta/3}\right);
\end{eqnarray}
and
\begin{eqnarray}\label{eq3:thm:Bounded Tensor Bernstein}
\mathrm{Pr} \left( \lambda_{\max}\left( \sum\limits_{i=1}^{n} \mathcal{X}_i \right)\geq \theta \right) \leq \mathbb{I}_1^M \exp \left( \frac{-3 \theta^2}{ 8 \sigma^2}\right)~~\mbox{for $\theta \leq \sigma^2/T$};
\end{eqnarray}
and
\begin{eqnarray}\label{eq4:thm:Bounded Tensor Bernstein}
\mathrm{Pr} \left( \lambda_{\max}\left( \sum\limits_{i=1}^{n} \mathcal{X}_i \right)\geq \theta \right) \leq \mathbb{I}_1^M \exp \left( \frac{-3 \theta}{ 8 T } \right)~~\mbox{for $\theta \geq \sigma^2/T$}.
\end{eqnarray}
\end{theorem}
\begin{proof}
Without loss of generality, we can assume that $T=1$ since the summands are 1-homogeneous and the variance is 2-homogeneous. From Lemma~\ref{lam:Bounded Bernstein mgf}, we have 
\begin{eqnarray}\label{eq5:thm:Bounded Tensor Bernstein}
\mathbb{E}e^{t \mathcal{X}_i} \preceq e^{(e^t - t- 1) \mathbb{E}(\mathcal{X}_i^2)} \mbox{~~ for $t > 0$.}
\end{eqnarray}
By applying Corollary~\ref{cor:3_7}, we then have 
\begin{eqnarray}\label{eq6:thm:Bounded Tensor Bernstein}
\mathrm{Pr} \left( \lambda_{\max}\left( \sum\limits_{i=1}^{n} \mathcal{X}_i \right)\geq \theta \right) &\leq &\mathbb{I}_1^{M} \exp\left(- t \theta + (e^t - t - 1) \lambda_{\max}\left( \sum\limits_{i=1}^n \mathbb{E} \left( \mathcal{X}^2_i \right)\right)\right) \nonumber \\
&=& \mathbb{I}_1^{M} \exp \left( - t \theta + \sigma^2(e^t - t - 1)   \right).
\end{eqnarray}
The right-hand side of Eq.~\eqref{eq6:thm:Bounded Tensor Bernstein} can be minimized by setting $t = \log (1 + \theta/\sigma^2)$. Substitute such $t$ and simplify the right-hand side of Eq.~\eqref{eq6:thm:Bounded Tensor Bernstein}, we obtain  Eq.~\eqref{eq2:thm:Bounded Tensor Bernstein}.

For $\theta \leq \sigma^2/T$, we have 
\begin{eqnarray}
\frac{1}{\sigma^2 + T\theta/3} \geq \frac{1}{\sigma^2 + T (\sigma^2/T) /3} = \frac{3}{4\sigma^2}, 
\end{eqnarray}
then, we obtain  Eq.~\eqref{eq3:thm:Bounded Tensor Bernstein}. Correspondingly, for $\theta \geq \sigma^2/T$, we have 
\begin{eqnarray}
\frac{\theta}{\sigma^2 + T\theta/3} \geq \frac{\sigma^2/T}{\sigma^2 + T (\sigma^2/T)/3} = \frac{3}{4 T}, 
\end{eqnarray}
and, we obtain  Eq.~\eqref{eq4:thm:Bounded Tensor Bernstein} also. 
\end{proof}

The following theorem~\ref{thm:Subexponential Tensor Bernstein} is the extension of the theorem~\ref{thm:Bounded Tensor Bernstein} by allowing the moments of the random tensors to grow at a controlled rate. We have to prepare subexponential Bernstein moment-generating function Lemma first for later proof of Theorem~\ref{thm:Subexponential Tensor Bernstein}

\begin{lemma}\label{lem:Subexponential Bernstein mgf}
Suppose that $\mathcal{X}$ is a random Hermitian tensor that satisfies
\begin{eqnarray}\label{eq1:lem:Subexponential Bernstein mgf}
\mathbb{E} \mathcal{X} = 0 \mbox{~~and~~} \mathbb{E}(\mathcal{X}^p)  \preceq \frac{p! \mathcal{A}^2}{2} 
\mbox{for $p=2,3,4,\cdots$.} 
\end{eqnarray}
Then, we have
\begin{eqnarray}
 \exp\left(\frac{t^2 \mathcal{A}^2}{2(1 - t)}\right)   \succeq  \mathbb{E}e^{t \mathcal{X}},
\end{eqnarray}
where $0 < t < 1$.
\end{lemma}
\begin{proof}
From Tayler series of the tensor exponential expansion, we have 
\begin{eqnarray}
\mathbb{E}e^{t \mathcal{X}} &=& \mathcal{I} + t \mathbb{E}\mathcal{X} + \sum\limits_{p=2}^{\infty} \frac{t^p \mathbb{E}(\mathcal{X}^p)}{p!} \preceq \mathcal{I}+ \sum\limits_{p=2}^{\infty} \frac{t^p \mathcal{A}^2}{2} \nonumber \\
&=& \mathcal{I} + \frac{t^2 \mathcal{A}^2}{2 (1-t)} \preceq \exp\left(\frac{t^2 \mathcal{A}^2}{2(1 - t)} \right),
\end{eqnarray}
therefore, this Lemma is proved. 
\end{proof}

\begin{theorem}[Subexponential Tensor Bernstein Bounds]\label{thm:Subexponential Tensor Bernstein}
Given a finite sequence of independent Hermitian tensors $\{ \mathcal{X}_i  \in \mathbb{C}^{I_1 \times \cdots \times I_M  \times I_1 \times \cdots \times I_M } \}$ that satisfy
\begin{eqnarray}\label{eq1:thm:Subexponential Tensor Bernstein}
\mathbb{E} \mathcal{X}_i = 0 \mbox{~~and~~} \mathbb{E} (\mathcal{X}^p_i) \preceq \frac{p! T^{p-2} }{2} \mathcal{A}_i^2,
\end{eqnarray}
where $p = 2,3,4,\cdots$. 

Define the total varaince $\sigma^2$ as: $\sigma^2 \define \left\Vert \sum\limits_i^n \mathcal{A}_i^2 \right\Vert$.
Then, we have following inequalities:
\begin{eqnarray}\label{eq2:thm:Subexponential Tensor Bernstein}
\mathrm{Pr} \left( \lambda_{\max}\left( \sum\limits_{i=1}^{n} \mathcal{X}_i \right)\geq \theta \right) \leq \mathbb{I}_1^M \exp \left( \frac{-\theta^2/2}{\sigma^2 + T\theta}\right);
\end{eqnarray}
and
\begin{eqnarray}\label{eq3:thm:Subexponential Tensor Bernstein}
\mathrm{Pr} \left( \lambda_{\max}\left( \sum\limits_{i=1}^{n} \mathcal{X}_i \right)\geq \theta \right) \leq \mathbb{I}_1^M \exp \left( \frac{-\theta^2}{ 4 \sigma^2}\right)~~\mbox{for $\theta \leq \sigma^2/T$};
\end{eqnarray}
and
\begin{eqnarray}\label{eq4:thm:Subexponential Tensor Bernstein}
\mathrm{Pr} \left( \lambda_{\max}\left( \sum\limits_{i=1}^{n} \mathcal{X}_i \right)\geq \theta \right) \leq \mathbb{I}_1^M \exp \left( \frac{- \theta}{ 4 T } \right)~~\mbox{for $\theta \geq \sigma^2/T$}.
\end{eqnarray}
\end{theorem}
\begin{proof}
Without loss of generality, we can assume that $T=1$. From Lemma~\ref{lem:Subexponential Bernstein mgf}, we have 
\begin{eqnarray}
\mathbb{E}e^{t \mathcal{X}_i} \preceq e^{\frac{t^2 \mathcal{A}^2_i}{2(1-t)}}, 
\end{eqnarray}
where $0 < t < 1$.

By applying Corollary~\ref{cor:3_7}, we then have 
\begin{eqnarray}\label{eq6:thm:Subexponential Tensor Bernstein}
\mathrm{Pr} \left( \lambda_{\max}\left( \sum\limits_{i=1}^{n} \mathcal{X}_i \right)\geq \theta \right) &\leq &\mathbb{I}_1^{M} \exp\left(- t \theta + \frac{t^2}{2(1-t)} \lambda_{\max}\left( \sum\limits_{i=1}^n \mathcal{A}_i^2  \right)\right) \nonumber \\
&=& \mathbb{I}_1^{M} \exp \left( - t \theta + \frac{\sigma^2 t^2}{2(1-t)}  \right).
\end{eqnarray}
The right-hand side of Eq.~\eqref{eq6:thm:Subexponential Tensor Bernstein} can be minimized by setting $t = \frac{\theta}{\theta + \sigma^2}$. Substitute such $t$ and simplify the right-hand side of Eq.~\eqref{eq6:thm:Subexponential Tensor Bernstein}, we obtain  Eq.~\eqref{eq2:thm:Subexponential Tensor Bernstein}.

For $\theta \leq \sigma^2/T$, we have 
\begin{eqnarray}
\frac{1}{\sigma^2 + T\theta} \geq \frac{1}{\sigma^2 + T (\sigma^2/T )} = \frac{1}{2\sigma^2}, 
\end{eqnarray}
then, we obtain  Eq.~\eqref{eq3:thm:Subexponential Tensor Bernstein}. Similarly, for $\theta \geq \sigma^2/T$, we have 
\begin{eqnarray}
\frac{\theta}{\sigma^2 + T\theta} \geq \frac{\sigma^2/T}{\sigma^2 + T( \sigma^2/T) } = \frac{1}{2T}, 
\end{eqnarray},
therefore, we also obtain  Eq.~\eqref{eq4:thm:Subexponential Tensor Bernstein}. 
\end{proof}

\subsection{Application of Tensor Bernstein Bounds}\label{sec:Application of Tensor Bernstein Bounds}

The tensor Bernstein bounds can also be extended to rectangular tensors by dilation. Consider a sequence of tensors $\{\mathcal{Y}_i\} \in \mathbb{C}^{I_1 \times \cdots \times I_M \times J_1 \times \cdots \times J_M}$ satisfy following:
\begin{eqnarray}
\mathbb{E} \mathcal{Y}_i = \mathcal{O} \mbox{~~and~~} \left\Vert \mathcal{Y}_i \right\Vert \leq T \mbox{~~almost surely.}
\end{eqnarray}
If the variance $\sigma^2$ is expressed as:
\begin{eqnarray}
\sigma^2 \define  \max \Bigg\{ \left\Vert \sum\limits_{i=1}^n \mathcal{Y}_i \star_M \mathcal{Y}^H_i \right\Vert,  \left\Vert \sum\limits_{i=1}^n \mathcal{Y}^H_i \star_M \mathcal{Y}_i \right\Vert\Bigg\},
\end{eqnarray}
we have 
\begin{eqnarray}
\mathrm{Pr} \left( \left\Vert \sum\limits_{i=1}^{n} \mathcal{Y}_i \right\Vert \geq \theta \right) \leq \prod_{m=1}^M (I_m + J_m) \exp \left( \frac{-\theta^2/2}{\sigma^2 + T\theta/3}\right);
\end{eqnarray}
from Theorem~\ref{thm:Bounded Tensor Bernstein}. 

Another application of tensor Bernstein bounds is to get upper and lower Bounds for the maximum eigenvalue with subexponential tensors. This application can relax the corollary~\ref{cor:Bounds for the Maximum Eigenvalue} conditions by allowing the moments of the random tensors to grow at a controlled rate. 
\begin{corollary}[Upper and Lower Bounds for the Maximum Eigenvalue for Subexponential Tensors]\label{cor:Bounds for the Maximum Eigenvalue Subexponential Tensors}
Consider a sequence $\{ \mathcal{X}_i  \in \mathbb{C}^{I_1 \times \cdots \times I_M  \times I_1 \times \cdots \times I_M } \}$ of independent, random, Hermitian tensors that satisfy
\begin{eqnarray}
\mathcal{X}_i \succeq \mathcal{O} \mbox{~~and~~} \mathbb{E} (\mathcal{X}^p_i) \preceq \frac{p! T^{p-2} }{2} \mathcal{A}_i^2,
\end{eqnarray}
and $\sigma^2 \define \left\Vert \sum\limits_{i=1}^n \mathcal{A}^2_i \right\Vert$. Then, we have 
\begin{eqnarray}\label{eq1:cor:Bounds for the Maximum Eigenvalue Subexponential Tensors}
\mu_{\max} \leq \mathbb{E} \lambda_{\max}\left( \sum\limits_{i=1}^n \mathcal{X}_i \right) \leq
2\mathbb{I}_1^M\left(\sigma \mathfrak{G}\left(\frac{\sigma}{2T} \right) + 2T e^{-\frac{\sigma^2}{4T^2}}\right),
\end{eqnarray}
where $\mathfrak{G}(\frac{\sigma}{2T}) \define \int_0^{\frac{\sigma}{2T}} e^{-s^2} ds$.
\end{corollary}
\begin{proof}
The lower bound at Eq.~\eqref{eq1:cor:Bounds for the Maximum Eigenvalue Subexponential Tensors} is true from the convexity of the function $\mathcal{A} \rightarrow \lambda_{\max}(\mathcal{A})$ and the Jensen's inequality. 

For the upper bound, we have 
\begin{eqnarray}\label{eq2:cor:Bounds for the Maximum Eigenvalue Subexponential Tensors}
\mathbb{E} \lambda_{\max}\left( \sum\limits_{i=1}^n \mathcal{X}_i \right)
&=& \int_{0}^{\infty} \mathrm{Pr} \left( \lambda_{\max}\left( \sum\limits_{i=1}^n \mathcal{X}_i \right)  \geq t \right) d t  \nonumber \\
&\leq_1& \mathbb{I}_1^{M}  \int_{0}^{\frac{\sigma^2}{T}}\exp \left(- \frac{t^2}{4 \sigma^2} \right)   dt +  \mathbb{I}_1^{M}  \int_{\frac{\sigma^2}{T}}^{\infty}\exp \left(- \frac{t}{4 T} \right)   dt \nonumber \\
&=&  2\mathbb{I}_1^M\left(\sigma \mathfrak{G}\left(\frac{\sigma}{2T} \right) + 2T e^{-\frac{\sigma^2}{4T^2}}\right),
\end{eqnarray}
where the inequality $\leq_1$ comes from the Eqs.~\eqref{eq3:thm:Subexponential Tensor Bernstein} and~\eqref{eq4:thm:Subexponential Tensor Bernstein}. This corollary is proved by introducing Gaussian integral function $\mathfrak{G}(x) \define \int_0^{x} e^{-s^2} ds$.  
\end{proof}

\section{Martingale Deviation Bounds}\label{sec:Martingale Deviation Bounds}

In this section, we introduce concepts about tensor martingales in Section~\ref{sec:Tensor Martingales}, and extend Hoeffding, Azuma, and McDiarmid inequalities to tensors context in Section~\ref{sec:Tensor Martingale Deviation Bounds}. 

\subsection{Tensor Martingales}\label{sec:Tensor Martingales}

Necessary definitions about tensor martingales will be provided here for later tensor martingale deviation bounds derivations. Let $(\Omega, \mathfrak{F}, \mathbb{P})$ be a master probability space. Consider a filtration $\{ \mathfrak{F}_i \}$ contained in the master sigma algebra as:
\begin{eqnarray}
\mathfrak{F}_0 \subset \mathfrak{F}_1 \subset \mathfrak{F}_2 \subset \cdots 
\subset \mathfrak{F}_{\infty} \subset \mathfrak{F}.
\end{eqnarray}
Given such a filtration, we define the conditional expectation $\mathbb{E}_i[ \cdot ] \define \mathbb{E}_i[ \cdot | \mathfrak{F}_i]$. A sequence $\{ \mathcal{Y}_i \}$ of random tensors is called \emph{adapted} to the filtration when each tensor $\mathcal{Y}_i$ is measurable with respect to $\mathfrak{F}_i$. We can think that an adapted sequence is one where the present depends only on the past. 

An adapted sequence $\{ \mathcal{X}_i \}$ of Hermitian tensors is named as a \emph{tensor martingale} when 
\begin{eqnarray}
\mathbb{E}_{i-1} \mathcal{X}_i = \mathcal{X}_{i-1} \mbox{~~~and~~~} \mathbb{E} \left\Vert \mathcal{X}_i \right\Vert < \infty,
\end{eqnarray}
where $i = 1, 2, 3, \cdots$. We obtain a scalar martingale if we track any fixed entry of a tensor martingale $\{ \mathcal{X}_i \}$. Given a tensor martingale $\{ \mathcal{X}_i \}$, we can construct the following new sequence of tensors 
\begin{eqnarray}
\mathcal{Y}_i \define \mathcal{X}_i  - \mathcal{X}_{i-1}  \mbox{~~for $i=1, 2, 3,\cdots$}
\end{eqnarray}
We then have $\mathbb{E}_{i-1}\mathcal{Y}_i = \mathcal{O}$.

\subsection{Tensor Martingale Deviation Bounds}\label{sec:Tensor Martingale Deviation Bounds}

Two Lemmas should be presented first before presenting tensor martingale deviation bounds and their proofs.

\begin{lemma}[Tensor Symmetrization]\label{lma:tensor symmetriation}
Let $\mathcal{A}$ be a fixed Hermitian tensor, and let $\mathcal{X}$ be a random Hermitian tensor with $\mathbb{E}\mathcal{E} = \mathcal{O}$. Then
\begin{eqnarray}
\mathbb{E} \mathrm{Tr}e^{\mathcal{A} + \mathcal{X}} \leq \mathbb{E} \mathrm{Tr}e^{\mathcal{A} + 2\beta \mathcal{X}},
\end{eqnarray}
where $\beta$ is Rademacher random variable.
\end{lemma}
\begin{proof}
Build an independent copy random tensor $\mathcal{Y}$ from $\mathcal{X}$, and let $\mathbb{E}_{\mathcal{Y}}$ denote the expectation with respect to the new random tensor $\mathcal{Y}$. Then, we have
\begin{eqnarray}
\mathbb{E} \mathrm{Tr}e^{\mathcal{A} + \mathcal{X}} =
\mathbb{E} \mathrm{Tr}e^{\mathcal{A} + \mathcal{X} - \mathbb{E}_{\mathcal{Y}} \mathcal{Y}}  \leq \mathbb{E} \mathrm{Tr}e^{\mathcal{A} + \mathcal{X} - \mathcal{Y}}
= \mathbb{E} \mathrm{Tr}e^{\mathcal{A} + \beta( \mathcal{X} - \mathcal{Y})},
\end{eqnarray}
where the first equality uses $\mathbb{E}_{\mathcal{Y}} \mathcal{Y} = \mathcal{O}$; the inequality uses the convexity of the trace exponential with Jensen's inequality; finally, the last equality comes from that the random tensor $\mathcal{X} - \mathcal{Y}$ is a symmetric random tensor and Rademacher is also a symmetric random variable.

This Lemma is established by the following:
\begin{eqnarray}
\mathbb{E} \mathrm{Tr}e^{\mathcal{A} + \mathcal{X}} &\leq& \mathbb{E} \mathrm{Tr} \left( e^{\mathcal{A}/2 + \beta \mathcal{X}}  e^{\mathcal{A}/2 - \beta \mathcal{Y}}\right) \nonumber \\
& \leq & \left( \mathbb{E} \mathrm{Tr} e^{\mathcal{A} + 2 \beta \mathcal{X}}\right)^{1/2}  \left( \mathbb{E} \mathrm{Tr} e^{\mathcal{A} - 2 \beta \mathcal{Y}}\right)^{1/2}
= \mathbb{E} \mathrm{Tr}e^{\mathcal{A} + 2 \beta \mathcal{X}},
\end{eqnarray}
where the first inequality comes from tensor Golden-Thompson inequality~\cite{chang2020tensor}, the second inequality comes from the Cauchy-Schwarz inequality, and the last identity follows from that the two factors are identically distributed. 
\end{proof} 

The other Lemma is to provide the tensor cumulant-generating function of a symetrized random tensor. 
\begin{lemma}[Cumulant-Generating Function of Symetrized Random tensor]\label{lma:Azuma CGF}
Given that $\mathcal{X}$ is a random Hermitian tensor and $\mathcal{A}$ is a fixed Hermitian tensor that satisfies $\mathcal{X}^2 \preceq \mathcal{A}^2$. Then, we have
\begin{eqnarray}
\log \mathbb{E} \left [ e^{2 \beta t  \mathcal{X}} | \mathcal{X} \right] \preceq 2 t^2 \mathcal{A}^2, 
\end{eqnarray}
where $\beta$ is a Rademacher random variable.
\end{lemma}
\begin{proof}
From Lemma~\ref{lma:mgf of Rademacher and normal rvs.}, we have 
\begin{eqnarray}
\mathbb{E} \left [ e^{2 \beta t \mathcal{X}} | \mathcal{X} \right] \preceq  e^{2 t ^2 \mathcal{X}^2}.
\end{eqnarray}
And, from the monotone property of logarithm, we also have 
\begin{eqnarray}
\log \mathbb{E} \left [ e^{2 t \theta \mathcal{X}} | \mathcal{X} \right] \preceq  2t^2 \mathcal{X}^2 \preceq 2 t^2 \mathcal{A}^2~~\mbox{for $t \in \mathbb{R}$}.  
\end{eqnarray}
Therefore, this Lemma is proved. 
\end{proof}

At this point, we are ready to present tensor martingale deviation bounds. In probability theory, the Azuma inequality for a scaler martingale gives normal concentration about its mean value, and the deviation is controlled by the total maximum squared of the difference sequence. Below is a tensor extension.

% Tensor Azuma
\begin{theorem}[Tensor Azuma Inequality]\label{thm:Tensor Azuma}
Given a finite adapted sequence of Hermitian tensors $\{ \mathcal{X}_i  \in \mathbb{C}^{I_1 \times \cdots \times I_M  \times I_1 \times \cdots \times I_M } \}$ and a fixed sequence of Hermitian tensors $\{ \mathcal{A}_i \}$ that satisfy
\begin{eqnarray}\label{eq1:thm:Tensor Azuma}
\mathbb{E}_{i-1} \mathcal{X}_i = 0 \mbox{~~and~~} \mathcal{X}^2_i \preceq  \mathcal{A}^2_i \mbox{almost surely}, 
\end{eqnarray}
where $i = 1,2,3,\cdots$. 

Define the total varaince $\sigma^2$ as: $\sigma^2 \define \left\Vert \sum\limits_i^n \mathcal{A}_i^2 \right\Vert$.
Then, we have following inequalities:
\begin{eqnarray}\label{eq2:thm:Tensor Azuma}
\mathrm{Pr} \left( \lambda_{\max}\left( \sum\limits_{i=1}^{n} \mathcal{X}_i \right)\geq \theta \right) \leq \mathbb{I}_1^M e^{-\frac{\theta^2}{8 \sigma^2}}.
\end{eqnarray}
\end{theorem}
\begin{proof}
Define the filtration $\mathfrak{F}_i \define \mathfrak{F}(\mathcal{X}_1, \cdots, \mathcal{X}_i)$ for the process $\{\mathcal{X}_i \}$. Then, we have 
\begin{eqnarray}\label{eq3:thm:Tensor Azuma}
\mathbb{E} \mathrm{Tr} \exp\left( \sum\limits_{i=1}^n t \mathcal{X}_i \right) &=& \mathbb{E}   \left( \mathbb{E}\left( \mathrm{Tr}\exp\left( \sum\limits_{i=1}^{n-1} t \mathcal{X}_i+ t \mathcal{X}_n \right) | \mathfrak{F}_{n}\right) | \mathfrak{F}_{n-1}\right) \nonumber \\
& \leq & \mathbb{E}   \left( \mathbb{E}\left( \mathrm{Tr}\exp\left( \sum\limits_{i=1}^{n-1} t \mathcal{X}_i+ 2 \beta t \mathcal{X}_n \right) | \mathfrak{F}_{n}\right) | \mathfrak{F}_{n}\right) \nonumber \\
& \leq &  \mathbb{E}   \left( \mathrm{Tr}\exp\left( \sum\limits_{i=1}^{n-1} t \mathcal{X}_i+ \log \mathbb{E} \left( e^{2 \beta t \mathcal{X}_n} | \mathfrak{F}_{n}\right)  \right) | \mathfrak{F}_{n}\right) \nonumber \\
& \leq & \mathbb{E} \mathrm{Tr} \exp \left( \sum\limits_{i=1}^{n-1} t \mathcal{X}_i +2 t^2 \mathcal{A}_n^2 \right),
\end{eqnarray}
where the first equality comes from the total expectation property of conditional expectation; the first inequality comes from Lemma~\ref{lma:tensor symmetriation}; the second inequality comes from Corollary~\ref{cor:3.3} and the relaxation the connection to the larger algebra set $\mathfrak{F}_n$; finally, the last inequality requires Lemma~\ref{lma:Azuma CGF}.

If we continue the iteration procedure based on Eq.~\eqref{eq3:thm:Tensor Azuma}, we have
\begin{eqnarray}\label{eq4:thm:Tensor Azuma}
\mathbb{E} \mathrm{Tr} \exp \left( \sum\limits_{i=1}^n t \mathcal{X}_i  \right) 
\leq \mathrm{Tr} \exp \left( 2t^2 \sum\limits_{i=1}^n \mathcal{A}^2_i \right),
\end{eqnarray}
then apply Eq.~\eqref{eq4:thm:Tensor Azuma} into Lemma~\ref{lma: Laplace Transform Method}, we obtain 
\begin{eqnarray}
\mathrm{Pr} \left( \lambda_{\max}\left( \sum\limits_{i=1}^{n} \mathcal{X}_i \right)\geq \theta \right) &\leq&  \inf\limits_{t > 0} \Big\{ e^{- t \theta} \mathbb{E} \mathrm{Tr}\exp \left( \sum\limits_{i=1}^n t \mathcal{X}_i \right) \Big\} \nonumber \\
& \leq &  \inf\limits_{t > 0} \Big\{ e^{- t \theta} \mathbb{E} \mathrm{Tr}\exp \left( 2 t^2 \sum\limits_{i=1}^n  \mathcal{A}^2_i \right) \Big\} \nonumber \\
& \leq & \inf\limits_{t > 0} \Big\{ e^{- t \theta} \mathbb{I}_{1}^M 
\lambda_{\max} \left(  \exp \left( 2t^2 \sum\limits_{i=1}^{n} \mathcal{A}^2_i \right)\right) \Big\} \nonumber \\
& = &\inf\limits_{t > 0} \Big\{ e^{- t \theta} \mathbb{I}_{1}^M 
\exp \left( 2t^2 \sigma^2 \right) \Big\} \nonumber \\
& \leq &  \mathbb{I}_{1}^M  e^{-\frac{\theta^2}{8 \sigma^2}},
\end{eqnarray}
where the third inequality utilizes $\lambda_{\max}$ to bound trace, the equality applies the definition of $\sigma^2$ and spectral mapping theorem, finally, we select $t = \frac{\theta}{4 \sigma^2}$ to minimize the upper bound to obtain this theorem.
\end{proof}

If we add extra assumption that the summands are independent, Theorem~\ref{thm:Tensor Azuma} gives a tensor extension of Hoeffding’s inequality. If we apply Theorem~\ref{thm:Tensor Azuma} to a Hermitian tensor martingale, we will have following corollary.
\begin{corollary}\label{cor:7.2}
Given a Hermitian tensor martingale $\{\mathcal{Y}_i: i=1,2,\cdots,n \} \in $ \\ $\mathbb{C}^{I_1 \times \cdots \times I_M \times I_1 \times \cdots \times I_M}$, and let $\mathcal{X}_i$ be the difference sequence of $\{\mathcal{Y}_i \}$, i.e., $\mathcal{X}_i \define \mathcal{Y}_i - \mathcal{Y}_{i-1}$ for $i = 1,2,3, \cdots$. If the difference sequence satisfies 
\begin{eqnarray}\label{eq1:cor:7.2}
\mathbb{E}_{i-1} \mathcal{X}_i = 0 \mbox{~~and~~} \mathcal{X}^2_i \preceq  \mathcal{A}_i \mbox{almost surely}, 
\end{eqnarray}
where $i = 1,2,3,\cdots$ and the total varaince $\sigma^2$ is defined as as: $\sigma^2 \define \left\Vert \sum\limits_i^n \mathcal{A}_i^2 \right\Vert$. Then, we have
\begin{eqnarray}
\mathrm{Pr} \left( \lambda_{\max}\left( \mathcal{Y}_n -  \mathbb{E} \mathcal{Y}_n \right)\geq \theta \right) \leq \mathbb{I}_1^M e^{-\frac{\theta^2}{8 \sigma^2}}.
\end{eqnarray}
\end{corollary}

% Tensor Hoeffding, special from Tensor Azuma as independent 

% Tensor McDiarmid
In the scalar setting, McDiarmid inequality can be treated as a corollary of Azuma’s inequality. McDiarmid inequality states that a function of independent random variables exhibits normal concentration about its mean, and the variance
depends on the function value sensitivity with respect to the input. A version
of the bounded differences inequality is still valid in tensor context.

\begin{theorem}[Tensor McDiarmid Inequality]\label{thm:Tensor McDiarmid}
Given a set of $n$ independent random variables, i.e. $\{X_i: i = 1,2,\cdots n\}$, and let 
$F$ be a Hermitian tensor-valued function that maps these $n$ random variables to a Hermitian tensor of dimension within $\mathbb{C}^{I_1 \times \cdots \times I_M  \times I_1 \times \cdots \times I_M }$. Consider a sequence of Hermitian tensors $\{ \mathcal{A}_i \}$ that satisfy
\begin{eqnarray}\label{eq1:thm:Tensor McDiarmid}
\left( F(x_1,\cdots,x_i,\cdots,x_n)  -  F(x_1,\cdots,x'_i,\cdots,x_n) \right)^2 \preceq \mathcal{A}^2_i,
\end{eqnarray}
where $x_i, x'_i \in X_i$ and $1 \leq i \leq n$. Define the total variance $\sigma^2$ as: $\sigma^2 \define \left\Vert \sum\limits_i^n \mathcal{A}_i^2 \right\Vert $.
Then, we have following inequality:
% x-z, y 
\begin{eqnarray}\label{eq2:thm:Tensor McDiarmid}
\mathrm{Pr} \left( \lambda_{\max}\left( F(x_1,\cdots,x_n) - \mathbb{E}F(x_1,\cdots,x_n )\right)\geq \theta \right) \leq \mathbb{I}_1^M e^{-\frac{\theta^2}{8 \sigma^2}}.
\end{eqnarray}
\end{theorem}
\begin{proof}
We define following random tensors $\mathcal{Y}_i$ for $0 \leq i \leq n$ as:
\begin{eqnarray}\label{eq3:thm:Tensor McDiarmid}
\mathcal{Y}_i \define \mathbb{E} \left( F(x_1,\cdots,x_n) | X_1, \cdots, X_i \right) = \mathbb{E}_{X_{i+1}}  \mathbb{E}_{X_{i+2}} \cdots  \mathbb{E}_{X_{n}} F(x_1,\cdots,x_n),
\end{eqnarray}
where $ \mathbb{E}_{X_{i+1}}$ is the expectation with respect to the random variable $X_{i+1}$. The constructed sequence $\mathcal{Y}_i$ forms a martingale. The associated difference sequence with respect to $\mathcal{Y}$, denoted as $\{ \mathcal{Z}_i \}$, can be stated as:
\begin{eqnarray}\label{eq4:thm:Tensor McDiarmid}
\mathcal{Z}_i  \define \mathcal{Y}_i  - \mathcal{Y}_{i-1} = 
 \mathbb{E}_{X_{i+1}}  \mathbb{E}_{X_{i+2}} \cdots  \mathbb{E}_{X_{n}}\left(
 F(x_1,\cdots,x_n) -   \mathbb{E}_{X_{i}}  F(x_1,\cdots,x_n ) \right).
\end{eqnarray}

Because $(x_1,\cdots,x_i)$ forms a filtration with respect to $i$, we have 
\begin{eqnarray}
\mathbb{E}_{X_{i-1}}\mathcal{Y}_i &=& \mathbb{E}_{X_{i-1}} \left( \mathbb{E}_{X_{i+1}}  \mathbb{E}_{X_{i+2}} \cdots  \mathbb{E}_{X_{n}} F(x_1,\cdots,x_n) | X_{i-1} \right) \nonumber \\
&=& \mathbb{E}_{X_{i-1}} \left( \mathbb{E}_{X_{i}}  \mathbb{E}_{X_{i+1}} \cdots  \mathbb{E}_{X_{n}} F(x_1,\cdots,x_n) | X_{i-1} \right) = \mathbb{E}_{X_{i-1}}\mathcal{Y}_{i-1},
\end{eqnarray}
then,
\begin{eqnarray}\label{eq5:thm:Tensor McDiarmid}
\mathbb{E}_{X_{i-1}} \mathcal{Z}_i
&=& \mathbb{E}_{X_{i-1}}\mathcal{Y}_i - \mathbb{E}_{X_{i-1}}\mathcal{Y}_{i-1} = \mathcal{O}
\end{eqnarray}

Let $X'_i$ be an independent copy of $X_i$, and construct the following two random vectors:
\begin{eqnarray}
\mathbf{x}' &=&  (X_1,\cdots,X_{i-1},X'_i,X_{i+1},\cdots,X_n), \nonumber \\
\mathbf{x}  &=&  (X_1,\cdots,X_{i-1},X_i,X_{i+1},\cdots,X_n).
\end{eqnarray} 
Since $\mathbb{E}_{X_i} F(\mathbf{x}) =
\mathbb{E}_{X'_i} F(\mathbf{x}')$, we have 
\begin{eqnarray}\label{eq6:thm:Tensor McDiarmid}
\mathcal{Z}_i = \mathbb{E}_{X_{i+1}}  \mathbb{E}_{X_{i+2}} \cdots  \mathbb{E}_{X_{n}}  \mathbb{E}_{X'_{i}}  \left(
 F(\mathbf{x}) -    F(\mathbf{x}') \right).
\end{eqnarray}
Then, $\left(  F(\mathbf{x}) -    F(\mathbf{x}')  \right)^2 \preceq \mathcal{A}^2_i$ from requirement provided by Eq.~\eqref{eq1:thm:Tensor McDiarmid}. We have following upper bound 
\begin{eqnarray}\label{eq7:thm:Tensor McDiarmid}
\mathbb{E}_{X_{i+1}}  \mathbb{E}_{X_{i+2}} \cdots  \mathbb{E}_{X_{n}}  \mathbb{E}_{X'_{i}}  \left(
 F(\mathbf{x}) -    F(\mathbf{x}') \right)^2 \preceq \mathcal{A}^2_i.
\end{eqnarray}
Therefore, from conditions provided by Eq.~\eqref{eq5:thm:Tensor McDiarmid} and Eq.~\eqref{eq7:thm:Tensor McDiarmid}, this theorem is proved by applying Corollary~\ref{cor:7.2} to the martingale $\{\mathcal{Y}_i\}$.
\end{proof}

\section{Conclusion}\label{sec:Conclusion}

In this paper, we generalize Lapalce transform method and Lieb's concavity theorem from matrices to tensors, and apply these techniques to extend following classical bounds from the scalar to the tensor situation: Chernoff, Azuma, Hoeffding, Bennett, Bernstein, and McDiarmid. The purpose of these probability inequalities tries to identify large-deviation behavior of the extreme eigenvalue of the sums of random tensors. Tail bounds for the norm of a sum of random rectangular tensors follow as an immediate corollary. Finally, we also apply the proof techniques developed at this work to study tensor-valued martingales.  

\bibliographystyle{amsplain}
\bibliography{TailBoundsForTensorsSum_Bib}

\ACKNO{We are grateful to reviewers' precious comments.}

\end{document}